\documentclass[12pt]{amsart}

\usepackage{amsthm,amsfonts,amsmath,amssymb,latexsym,epsfig,mathrsfs,yfonts,marvosym}
\usepackage[usenames]{color}
\usepackage{graphicx}
\usepackage[all]{xy}
\usepackage{epsfig}
\usepackage{epic}
\usepackage{eepic}
\usepackage{hyperref}
\usepackage{dsfont}\let\mathbb\mathds
\usepackage[english]{babel}

\DeclareMathAlphabet\oldmathcal{OMS}        {cmsy}{b}{n}
\SetMathAlphabet    \oldmathcal{normal}{OMS}{cmsy}{m}{n}
\DeclareMathAlphabet\oldmathbcal{OMS}       {cmsy}{b}{n}
 
\usepackage{eucal}

\usepackage[active]{srcltx}

\newtheorem{theorem}{Theorem}[section]
\newtheorem{lemma}[theorem]{Lemma}
\newtheorem{proposition}[theorem]{Proposition}
\newtheorem{corollary}[theorem]{Corollary}

\newenvironment{remark}{\medskip \refstepcounter{theorem}
\noindent  {\bf Remark \thetheorem}.\rm}{\,}
\newenvironment{remarks}{\medskip \refstepcounter{theorem}
\noindent  {\bf Remarks \thetheorem}.\rm}{\,}
\newcommand     {\comment}[1]   {}
\newcommand{\mute}[2] {}
\newcommand     {\printname}[1] {}

\newtheorem*{ack}{Acknowledgements}

\def\d{\partial}

\def\<{\langle}

\def\>{\rangle}

\def\BOne{{\mathchoice {\rm 1\mskip-4mu l} {\rm 1\mskip-4mu l}
                          {\rm 1\mskip-4.5mu l} {\rm 1\mskip-5mu l}}}

\def\fract#1#2{\raise4pt\hbox{$ #1 \atop #2 $}}
\def\decdnar#1{\phantom{\hbox{$\scriptstyle{#1}$}}
\left\downarrow\vbox{\vskip15pt\hbox{$\scriptstyle{#1}$}}\right.}

\def\bbc{{\mathbb C}}

\def\bbp{{\mathbb P}}
\def\bbq{{\mathbb Q}}
\def\bbr{{\mathbb R}}

\def\bbz{{\mathbb Z}}

\def\gra{\alpha}
\def\grb{\beta}

\def\grd{\delta}

\def\grg{\gamma}

\def\grk{\kappa}
\def\grl{\lambda}

\def\gro{\omega}

\def\grr{\rho}

\def\grt{\tau}

\def\grz{\zeta}
\def\grD{\Delta}

\def\grG{\Gamma}

\def\cala{{\mathcal A}}

\def\calo{{\mathcal O}}

\def\cald{{\mathcal D}}

\def\calf{{\mathcal F}}

\def\calh{{\mathcal H}}
\def\cali{{\mathcal I}}

\def\calm{{\mathcal M}}
\def\caln{{\mathcal N}}
\def\calo{{\mathcal O}}

\def\cals{{\oldmathcal S}}

\def\calw{{\mathcal W}}

\def\la#1{\hbox to #1pc{\leftarrowfill}}
\def\ra#1{\hbox to #1pc{\rightarrowfill}}
\def\calz{{\oldmathcal Z}}

\def\ga{{\mathfrak a}}

\def\gc{{\mathfrak c}}

\def\gf{{\mathfrak f}}

\def\gh{{\mathfrak h}}
\def\gi{{\mathfrak i}}

\def\gm{{\mathfrak m}}
\def\gn{{\mathfrak n}}
\def\go{{\mathfrak o}}

\def\gr{{\mathfrak r}}
\def\gs{{\mathfrak s}}
\def\gt{{\mathfrak t}}
\def\gu{{\mathfrak u}}

\def\gy{{\mathfrak y}}
\def\gz{{\mathfrak z}}

\def\gA{{\mathfrak A}}
\def\gB{{\mathfrak B}}
\def\gC{{\mathfrak C}}
\def\gD{{\mathfrak D}}

\def\gH{{\mathfrak H}}

\def\gR{{\mathfrak R}}
\def\gS{{\mathfrak S}}
\def\gT{{\mathfrak T}}

\def\gX{{\mathfrak X}}

\def\hook{\mathbin{\hbox to 6pt{%
                 \vrule height0.4pt width5pt depth0pt
                 \kern-.4pt
                 \vrule height6pt width0.4pt depth0pt\hss}}}
\def\d{\partial}

\def\cX{\hat{X}}

\def\cJ{\hat{J}}
\def\cK{\hat{K}}

\def\12{\xi_{k_1,k_2}}
\def\m5{M^5_{k_1,k_2}}

\begin{document}

\title{Extremal Sasakian Geometry on $T^2\times S^3$ and Related Manifolds}

\author{Charles P. Boyer and Christina W. T{\o}nnesen-Friedman}\thanks{Christina T{\o}nnesen-Friedman was partially supported by a grant from the Simons Foundation (\#208799)}
\address{Charles P. Boyer, Department of Mathematics and Statistics,
University of New Mexico, Albuquerque, NM 87131.}
\email{cboyer@math.unm.edu} 
\address{Christina W. T{\o}nnesen-Friedman, Department of Mathematics, Union
College, Schenectady, New York 12308, USA } \email{tonnesec@union.edu}



\begin{abstract}
We prove the existence of extremal Sasakian structures occurring on a countably infinite number of distinct contact structures on $T^2\times S^3$ and certain related 5-manifolds. These structures occur in bouquets and exhaust the Sasaki cones in all except one case in which there are no extremal metrics.

\end{abstract}

\maketitle

\markboth{Extremal Sasakian Geometry on $T^2\times S^3$}{Charles P. Boyer and Christina W. T{\o}nnesen-Friedman}

\section{Introduction}
Little appears to be known about the existence of Sasakian structures on contact manifolds with non-trivial fundamental group outside those obtained by quotienting a simply connected Sasakian manifold by a finite group of Sasakian automorphisms acting freely. Although it has been known for sometime that $T^2\times S^3$ admits a contact structure \cite{Lut79}, it is unknown until now whether it admits a Sasakian structure. In the current paper we not only prove the existence of Sasakian structures on  $T^2\times S^3$, but also prove the existence of families, known as bouquets, of extremal Sasakian metrics on $T^2\times S^3$  as well as on certain related 5-manifolds.

We mention here that there is a toric contact structure on $T^2\times S^3$ \cite{Lut79,Ler02a}; however, its moment cone contains a two dimensional linear subspace. Thus, it follows from Proposition 8.4.38 of \cite{BG05} that this toric contact structure is not of Reeb type, and so cannot admit a compatible $T^3$-invariant Sasakian metric. Furthermore, $T^2\times S^3$ cannot admit {\it any} toric contact structure of Reeb type for the latter must have finite fundamental group \cite{Ler04}. Nevertheless, as we shall show, $T^2\times S^3$ does admit a countably infinite number of inequivalent contact structures $\cald_k,~k\in\bbz^+$ with a compatible $T^2$ action of Reeb type which fibers over the symplectic manifold $T^2\times S^2$. It is easy to see from Lutz that the toric contact structure on $T^2\times S^3$ has vanishing first Chern class; whereas, as we show below our $T^2$ invariant Sasakian structures do not. Hence, our contact structures are distinct from that of the toric case.

The organization of our paper proceeds as follows: in Section 2 we give the preliminaries of Sasakian geometry with emphasis on the Sasaki cone and Sasaki bouquet. In Section 3  we apply the join operation \cite{BGO06,BG05} to the three dimension nilmanifold $\caln^3$ and the three dimension sphere $S^3$ to determine the diffeomorphism type of our 5-manifolds\footnote{It is reasonable to expect a similar description of Sasakian geometry on the non-trivial $S^3$ bundle over $T^2$, but we have not done so here. It will, however, be treated in a forthcoming work \cite{BoTo12}}. The key here is a recent topological rigidity result of Kreck and L\"uck \cite{KrLu09}.  In Section 4 we turn to a brief review of the complex structures on ruled surfaces of genus one described by Suwa \cite{Suw69}, and in Section 5 we give a review of extremal K\"ahler structures on these surfaces based mainly on \cite{Fuj92,ACGT08}. In Section 6 we investigate Hamiltonian circle actions on $T^2\times S^2$. The important point is to describe Hamiltonian circle actions which represent distinct conjugacy classes of maximal tori. We are able to do this by computing rational homotopy groups using the recent work of Bu{\c{s}}e \cite{Bus10}. In Section 7 we describe the relevant Sasakian structures on $T^2\times S^3$ and certain related manifolds, and finally in Section 8 we prove our main results concerning the extremal Sasakian structures on these 5-manifolds by showing that in all but one case they exhaust the Sasaki cones. It is important to realize that deforming the Sasakian structure to obtain an extremal Sasaki metric also deforms the contact structure $\cald$. Thus, we are actually dealing with isotopy classes of contact structures that are obtained by isotopies that are invariant under the normalizer of the maximal torus in the CR automorphism group.  We denote such an isotopy class by $\bar{\cald}$, and prove that the Sasaki cone only depends on this isotopy class. 

Let $\calm$ denote the moduli space of complex structures on the torus $T^2$. We mention that for notational convenience we shall often suppress the dependence of objects on $\grt\in \calm$. Choosing a different complex structure $\grt'\in \calm$ has no effect whatsoever on the Sasaki cone, so this is why we suppress the notation. Often there are families of complex structures associated with each Sasaki cone. Let us now sketch the construction leading to our main theorem. As mentioned above our 5-manifolds $M^5_{k_1,k_2}$ are realized as circle bundles over $T^2\times S^2$. The underlying CR structures on $M^5_{k_1,k_2}$ are inherited from the complex structures of the base $T^2\times S^2$. These complex structures $J_{2m}$ arise by realizing $T^2\times S^2$ as projectivizations of rank two complex vector bundles over $T^2$ together with a choice of complex structure on $T^2$. These ruled surfaces are well understood from the work of Atiyah \cite{Ati55,Ati57} and Suwa \cite{Suw69}. When the chosen ruled surface admits a Hamiltonian Killing vector field, this vector field lifts to an infinitesimal automorphism of the induced Sasakian structure giving rise to a two dimensional Sasaki cone as described in \cite{BGS06}. Furthermore, as described in \cite{Boy10a,Boy10b} the Sasaki cones often come in bouquets associated to an isotopy class of contact structures which as mentioned above correspond to distinct conjugacy classes of tori in the contactomorphism group. Then by adapting the results of \cite{ACGT08} to the orbifold case and using the Openness Theorem of \cite{BGS06} we prove that in all but one case extremal metrics exhaust the Sasaki cone. It can also be mentioned that the construction of the extremal Sasakian metrics, given in Section \ref{extsassec}, is actually explicit. In this case the transversal K\"ahler structure admits a Hamiltonian $2$-form (see e.g. \cite{ACGT04} for a definition) or is a (local) product of a flat metric on $T^2$ and a metric with a Hamiltonian $2$-form.

\begin{theorem}\label{mainthm}
The contact manifolds $M^5_{k_1,k_2}=\caln^3\star_{k_1,k_2}S^3$ admit a bouquet of Sasakian structures for each $k_1\in \bbz^+$ and each positive integer $k_2$ relatively prime to $k_1$. These Sasaki bouquets consist of $\lceil\frac{k_1}{k_2}\rceil$ Sasaki cones $\grk(\bar{\cald}_{k_1},J_{2m})$ of dimension two with complex structures $J_{2m}$ labelled by $m=0,\cdots,\lceil\frac{k_1}{k_2}\rceil-1$,  plus a Sasaki cone $\grk(\bar{\cald}_{k_1},J)$ of dimension one where $J\in A_{0,\grt}$, the non-split complex structure. 
\begin{itemize} 
\item For each $m=0,\cdots,\lceil\frac{k_1}{k_2}\rceil-1$ extremal Sasakian structures exhaust the Sasaki cones $\grk(\bar{\cald}_{k_1},J_{2m})$.  Moreover, for $m=0$ there is a unique regular ray of extremal Sasakian structures with constant scalar curvature. 
\item For $J\in A_{0,\grt}$ the one dimensional Sasaki cone $\grk(\bar{\cald}_{k_1},J)$ admits no extremal Sasaki metric. 
\end{itemize}
Furthermore, $M^5_{k_1,1}$ is diffeomorphic to $T^2\times S^3$ for all $k_1\in\bbz^+$ and has a countably infinite number of distinct isotopy classes of contact structures $\bar{\cald}_{k_1}$ of Sasaki type.
\end{theorem}

\newpage
\begin{remarks}

\noindent 1. For the $m=0$ case we believe that the unique regular ray of constant scalar curvature metrics is actually the unique ray of constant scalar curvature metrics in each Sasaki cone.

\noindent 2. For $k_2>1$ the dependence of the homotopy type (and diffeomorphism type) of $M^5_{k_1,k_2}$ on $k_1$ is not understood at this time. Generally, they are lens space bundles over $T^2$.

\noindent 3. It can also be shown that for many cases with $m>0$, there exist constant scalar curvature extremal Sasaki metrics associated to a Reeb vector field in the Sasaki cone. We do not know at the present time whether these occur in all two dimensional Sasaki cones. This will be addressed in the sequel to this work \cite{BoTo12}.

\noindent 4. The Sasaki bouquet is complete with respect to a fixed contact form $\eta_{k_1,k_2}$ in the sense that there are no other Sasakian structures with contact form $\eta_{k_1,k_2}$ in the bouquet. Here we are including those Sasakian structures obtained by varying the transverse complex structure. For example, in the degree 0 case there is another $\bbc\bbp^1$'s worth of complex structures \cite{Suw69} giving $\calm\times \bbc\bbp^1$ as parameterizing the complex structures in this case. Moreover, in the degree $>0$ case there is the well-known jumping phenomenon \cite{MoKo06} as discussed briefly in Section \ref{complex structures}. This makes the moduli space of complex structures non-Hausdorff; hence, the moduli space of extremal Sasakian structures will also be non-Hausdorff.
\end{remarks}

\section{Preliminaries}
Here we give a brief review of Sasakian geometry referring to \cite{BG05} for details and further development. Sasakian  geometry can be thought of as the odd dimensional version of K\"ahlerian geometry. It consists of a smooth manifold $M$ of dimension $2n+1$ endowed with a contact 1-form together with a strictly pseudoconvex CR structure $(\cald,J)$. Explicitly it is given by a quadruple $\cals=(\xi,\eta,\Phi,g)$ where $\eta$ is a contact 1-form, $\xi$ is its Reeb vector field, $\Phi$ is an endomorphism field which annihilates $\xi$ and satisfies $J=\Phi |_\cald$ on the contact bundle $\cald=\ker\eta$. Moreover, $g$  is a Riemannian metric given by the equation 
\begin{equation}\label{sasmetric}
g=d\eta\circ(\Phi\otimes \BOne)+\eta\otimes\eta,
\end{equation}
and $\xi$ is a Killing vector field of $g$ which generates a one dimensional foliation $\calf_\xi$ of $M$ whose transverse structure is K\"ahler. There is a freedom of scaling, namely, given a Sasakian structure $\cals=(\xi,\eta,\Phi,g)$ consider the {\it transverse homothety} by sending the Sasakian structure $\cals=(\xi,\eta,\Phi,g)$ to $\cals_a=(a^{-1}\xi,a\eta,\Phi,g_a)$ 
where $a\in\bbr^+$ and 
$$g_a=ag+(a^2-a)\eta\otimes\eta.$$ 
This is another Sasakian structure which generally is inequivalent to $\cals$. Hence, Sasakian structures come in rays.

When $M$ is compact it follows from a theorem of Carri\`ere (cf. Theorem 2.6.4 of \cite{BG05}) that the closure of any leaf of $\calf_\xi$ is a torus $\gT$ of dimension at least one, and the flow is conjugate to a linear flow on the torus. This implies that for a dense subset of Sasakian structures $\cals$ on a compact manifold the leaves are all compact 1-dimensional manifolds, i.e circles. Such $\cals$ are known as {\it quasiregular} in which case the foliation $\calf_\xi$ comes from a locally free circle action. Then the quotient space $\calz$ has the structure of a projective algebraic orbifold with an induced K\"ahler form $\gro$ such that $\pi^*\gro=d\eta$ where $\pi$ is the quotient projection. If the circles comprising the leaves of $\calf_\xi$ all have the same period, $\cals$ is said to be {\it regular}, and the quotient space $\calz$ is a smooth projective algebraic variety with a trivial orbifold structure. The complex structure $\cJ$ on $\calz$ is also related to the CR structure $J$ on $M$. For any foliate vector field $X$ on $M$ we have $\pi_*\Phi X=\cJ\pi_*X$. We say that $J=\Phi |_\cald$ is the {\it horizontal lift} of $\cJ$.

Now the torus $\gT=\gT(\cals)$ lies in the group $\gA\gu\gt(\cals)$ of automorphisms of the Sasakian structure $\cals$. Letting $\gC\gR(\cald,J)$ denote the group of automorphisms of the CR structure $(\cald,J)$,  $\gC\go\gn(M,\cald)$ the Fr\'echet Lie group of contactomorphisms of $\cald$, and $\gC\go\gn(M,\eta)$ the Fr\'echet Lie subgroup consisting of elements in $\gC\go\gn(M,\cald)$ that leave the contact 1-form $\eta$ invariant, we have natural inclusions (including arrows)
$$\begin{matrix} &&&  \gC\gR(\cald,J) && \\
            &&\nearrow && \searrow & \\
            \gT\subset \gA\gu\gt(\cals)&&&&& \gC\go\gn(M,\cald) \\
            &&\searrow && \nearrow & \\
            &&& \gC\go\gn(M,\eta) &&
\end{matrix}.$$
It is known that $\gC\gR(\cald,J)$ is a compact Lie group except for the standard CR structure on $S^{2n+1}$ \cite{Lee96,Sch95} and that $\gC\go\gn(M,\eta)$ is a closed Fr\'echet Lie subgroup of $\gC\go\gn(M,\cald)$ \cite{Boy10a}.  Furthermore, $\gA\gu\gt(\cals)$ is a closed Lie subgroup of both $\gC\gR(\cald,J)$ and $\gC\go\gn(M,\eta)$. 

It is well known that for any contact 1-form $\eta$ the one dimensional Lie group $\gR_\xi$ generated by the Reeb vector field lies in the center of $\gC\go\gn(M,\eta)$ and hence when $\cals$ is Sasakian (or even K-contact), in the center of $\gA\gu\gt(\cals)$. However, $\gR_\xi$ is not necessarily a closed subgroup of either $\gA\gu\gt(\cals)$ nor $\gC\go\gn(M,\eta)$, but its closure is a torus $\gT_k$ of dimension $k\leq n+1$ which also lies in the center of both. Note that for any Sasakian structure we have $\dim~\gA\gu\gt(\cals)\geq 1$, and if strict inequality holds $\gA\gu\gt(\cals)$ must contain a 2-torus $\gT_2$. We are also interested in the Lie algebra of these groups which we denote with lower case gothic letters, viz. $\gt_k,\ga\gu\gt(\cals),\gc\gr(\cald,J),\gc\go\gn(M,\eta),\gc\go\gn(M,\cald)$. Given a contact structure $\cald$ with a fixed contact form $\eta$, a torus $\gT$ in $\gC\go\gn(M,\eta)$ is said to be of {\it Reeb type} \cite{BG00b,BG05} if the Reeb vector field $\xi$ of $\eta$ lies in the Lie algebra $\gt$ of $\gT$. In this paper we only consider torus actions of Reeb type.

\subsection{Sasaki Cones and the Sasaki Bouquet}
Let $(M^{2n+1},\cald)$ be a contact structure of Sasaki type. The Sasaki cone $\grk(\cald,J)$ was first defined in \cite{BGS06} to be the moduli space of Sasakian structures associated to a fixed underlying strictly pseudoconvex CR structure $(\cald,J)$. However, it is often convenient to fix a maximal torus $\gT_k(\cald,J)$ of Reeb type in the CR automorphism group $\gC\gR(\cald,J)$ and consider the `unreduced' Sasaki cone $\gt^+_k(\cald,J)$ defined to be the subset of all $X\in \gt_k(\cald,J)$ such that $\eta(X)>0$ where $\gt_k(\cald,J)$ denotes the Lie algebra of $\gT_k(\cald,J)$, $\eta$ is any contact form representing $\cald$, and $k$ denotes the dimension of the maximal torus. Then $\gt^+_k(\cald,J)$ is related to $\grk(\cald,J)$ by $\grk(\cald,J)=\gt_k^+(\cald,J)/\calw(\cald,J)$ where $\calw(\cald,J)$ is the Weyl group of $\gC\gR(\cald,J)$. Note that for a contact structure of Sasaki type on a compact manifold $1\leq k\leq n+1$, and $k=n+1$ is the toric case. Associated to a fixed oriented contact structure $\cald$ there are many compatible almost complex structures $J$, and some may be associated to K-contact or Sasakian structures. These give rise to bouquets $\gB(\cald)=\cup_\gra\grk(\cald,J_\gra)$ of Sasaki cones as described in \cite{Boy10a,Boy10b}. Generally, the Sasaki cones in a bouquet can have varying dimension (see Theorem 8.6 of \cite{Boy10a} for an example) and the index set can be arbitrary. For examples of Sasaki bouquets with toric Sasaki cones and finite index set on $S^2\times S^3$ see \cite{Boy10b,Boy11,BoPa10}. In the present paper the Sasaki cones of our bouquets occuring on $T^2\times S^3$ and certain related manifolds all have finite index set and in each bouquet all Sasaki cones but one has dimension 2. Generally, it is unknown whether or not Sasaki bouquets always have finite index set. A bouquet consisting of precisely $N$ Sasaki cones is called an {\it $N$-bouquet} and denoted by $\gB_N(\cald)$. In \cite{Boy10a} the index set of the bouquets were taken to be what was called the set of $T$-equivalence classes of almost complex structures that correspond to the same conjugacy class of maximal tori. Generally, there are large families of almost complex structures corresponding to the same conjugacy class of maximal tori; hence, there are families of Sasakian structures corresponding to the same Sasaki cone. So we can get moduli of Sasakian structures belonging to a fixed contact structure; however, as discussed in \cite{Boy10b} this moduli space can be non-Hausdorff. Indeed, this is the case in the present paper. An $N$-bouquet $\gB_N(\cald)$ is {\it complete} if $N$ is precisely the number of conjugacy classes of maximal tori in $\gC\go\gn(M,\cald)$, and it is {\it complete with respect to $\eta$} if $N$ is precisely the number of maximal tori in $\gC\go\gn(M,\eta)$. Notice that if $\gB_N(\cald)$ is complete with respect to $\eta$, then the intersection of the Sasaki cones in $\gB_N(\cald)$ contains the ray of the Reeb vector field $\xi$ of $\eta$.

\begin{remark}\label{sasconerem}
It is important to realize that a choice of Reeb vector field $\xi$ in a Sasaki cone $\grk(\cald,J)$ uniquely determines a Sasakian structure $\cals=(\xi,\eta,\Phi,g)$ since within a contact structure $\cald$ a Reeb vector field $\xi$ belongs to a unique contact form $\eta$, $\Phi$ is completely determined by $\xi$ and $J$, and the Sasaki metric $g$ is then determined by Equation (\ref{sasmetric}). As a consequence we often talk about a Sasakian structure being an element of the Sasaki cone $\grk(\cald,J)$.
\end{remark}

\subsection{The Join Construction}
Products of K\"ahlerian manifolds are K\"ahler, but products of Sasakian manifolds do not even have the correct dimension. Nevertheless, one can easily construct new Sasakian manifolds from old regular (or more generally quasi-regular) ones by constructing circle bundles over the product of K\"ahler manifolds (or orbifolds). This is the join construction as described in \cite{BGO06} and in Section 7.6.2 of \cite{BG05}. 
However, the Sasakian (or K-contact) structure is actually superfluous to the construction. It is natural to consider the join of quasi-regular contact manifolds; however, in this paper we only apply the join construction to regular contact structures. Let $M_i$ for $i=1,2$ be compact regular contact manifolds with Reeb vector fields $\xi_i$, respectively. These vector fields generate free circle actions on $M_i$ and the quotient manifolds are smooth symplectic manifolds $\calz_i$.  Then the quotient of the product $T^2=S^1\times S^1$ action on $M_1\times M_2$ is $\calz_1\times \calz_2$. Taking primitive symplectic forms $\gro_i$ on $\calz_i$ we consider the symplectic form $\gro_{k_1,k_2}=k_1\gro_1+k_2\gro_2$ on $\calz_1\times \calz_2$ where $k_1,k_2$ are relatively prime positive integers. Then by the Boothby-Wang construction the total space of the principal circle bundle over $\calz_1\times \calz_2$ corresponding to the cohomology class $[\gro_{k_1,k_2}]\in H^2(\calz_1\times \calz_2,\bbz)$ has a natural regular contact structure whose contact form $\eta_{k_1,k_2}$ satisfies $d\eta_{k_1,k_2}=\pi^*\gro_{k_1,k_2}$ where $\pi$ is the natural bundle projection. The total space of this bundle is denoted by $M_1\star_{k_1,k_2}M_2$ and is called {\it the join} of $M_1$ and $M_2$.

Choosing a complex structure (not necessarily the product structure) on the base $\calz_1\times \calz_2$ that makes $(\calz_1\times \calz_2,\gro_{k_1,k_2})$ a K\"ahler manifold, then gives rise via a well known construction \cite{BG05} to a Sasakian structure on the join $M_1\star_{k_1,k_2}M_2$.

\section{The Diffeomorphism Types}\label{secdiff}

We consider the join of  $S^3$ with its standard Sasakian structure and the nilmanifold $\caln^3$ constructed as the compact quotient of the Heisenberg group $\calh^3(\bbr)$ by its integral lattice $\calh^3(\bbz)$. The 3-dimensional Heisenberg group $\calh^3(\bbr)$ is given in coordinates by the nilpotent matrices of the form 
$$\left\{\left(
\begin{array}{ccc}
1 &x & z\\
0 & 1 & y\\
0 & 0 & 1 \end{array}\right)\  |\
x,y,z\in\bbr\right\}.$$
It has a natural bi-Sasakian structure \cite{Boy09}, and if we consider the nilmanifold $\caln^3$ to be the manifold of left cosets $\calh^3(\bbr)/\calh^3(\bbz)$, it inherits the right Sasakian structure from $\calh^3(\bbr)$. Actually it has a family of Sasakian structures coming from the family of underlying CR structures $(\cald,J_\grt)$. Now $\caln^3$ fibers over the 2-torus $T^2$ with its flat K\"ahlerian structures, and a result of Folland \cite{Fol04} says that there is a 1-1 correspondence between elements of the moduli space $\calm$ of complex structures on $T^2$ and the underlying CR structures on $\caln^3$. Hence, the moduli space $\calm$ parameterizes the standard Sasakian structures on $\caln^3$. These all have a transverse K\"ahler structure with a flat transverse metric. Thus, we have a family of inequivalent `standard' Sasakian structures $\cals_\grt=(\xi,\eta,\Phi_\grt,g)$ on $\caln^3$ that are equivalent as Riemannian structures, where $\grt\in \calm$. 

Next we determine the diffeomorphism type of $M^5_{k_1,1}=\caln^3\star_{k_1,1}S^3$ with their induced Sasakian structures and show that the 5-manifolds $M^5_{k_1,k_2}=\caln^3\star_{k_1,k_2}S^3$ with $k_2>1$ have a fundamental group that is a non-split central extension of $\bbz^2$ when $k_1,k_2$ are relatively prime positive integers. Explicitly, we shall prove

\begin{theorem}\label{diffeothm}
Let $M^5_{k_1,k_2}=\caln^3\star_{k_1,k_2} S^3$ be the regular Sasakian $(k_1,k_2)$-join of the nilmanifold $\caln^3$ with Sasakian structure $\cals_\grt$ and $S^3$ with its standard Sasakian structure where $\gcd(k_1,k_2)=1$. Then $M^5_{k_1,k_2}$ is an $L(k_2,1)$ lens space bundle over $T^2$ with $H_1(M^5_{k_1,k_2},\bbz)\approx \bbz^2$ and non-Abelian fundamental group when $k_2>1$.   Moreover, $\caln^3\star_{k_1,1} S^3$ is diffeomorphic to $T^2\times S^3$ for all $k_1\in\bbz^+$. However, when $k_2>1$, the fundamental group $\pi_1(M^5_{k_1,k_2})$ is a non-Abelian central extension of $\bbz^2$ by $\bbz_{k_2}$; hence, the lens space bundle is non-trivial. 
\end{theorem}

To prove this theorem, we first notice that $\caln^3\star_{k_1,k_2} S^3$ is a homogeneous manifold. This can be seen as follows: from the join construction we can write $\caln^3\star_{k_1,k_2} S^3$ as $(\caln^3\times S^3)/S^1(k_1,k_2)$ where the circle $S^1(k_1,k_2)$ is generated by the vector field $k_2\xi_1-k_1\xi_2$. The Reeb vector fields are given explicitly in coordinates $(x,y,z)$ on $\caln^3$ and $(z_1,z_2)$ on $\bbc^2$ by
$\xi_1=\d_z$ and $\xi_2$ is the restriction of the infinitesimal generator of the action $(z_1,z_2)\mapsto (e^{i\theta}z_1,e^{i\theta}z_2)$ to the unit sphere $S^3$ which we identify with the Lie group $SU(2)$ by 
$$(z_1,z_2)\longleftrightarrow \left(
\begin{matrix}z_1 & z_2\\
              -\bar{z}_2 & \bar{z}_1 
\end{matrix}
\right), \qquad |z_1|^2+|z_2|^2=1.$$
The  group $G=\calh^3(\bbr)\times SU(2)$ acts on $\caln^3\times S^3$ by the product action. Consider the subgroup $H$ of $G$ defined by
$$H=\{\left(
\begin{matrix}1&a&c+k_2t\\
              0&1&b \\
              0&0&1
\end{matrix}
\right)\times \left(
\begin{matrix}e^{-2\pi ik_1t}&0 \\
              0& e^{2\pi ik_1t}
\end{matrix}
\right) ~|~a,b,c\in \bbz,\quad t\in\bbr\}.
$$
It is a closed Lie subgroup and we have
\begin{lemma}\label{joinlem}
The homogeneous manifold $G/H$ can be identified with the join $\caln^3\star_{k_1,k_2} S^3$.
\end{lemma}

\begin{proof}
Consider the map $\psi:\calh^3(\bbz)\times \bbr\ra{1.6} H$ defined by 
$$\psi(N,t)=\Bigl(N\cdot \left(
\begin{matrix}1&0&k_2t\\
              0&1&0 \\
              0&0&1
\end{matrix}
\right)\Bigr)\times \left(
\begin{matrix}e^{-2\pi ik_1t}&0 \\
              0& e^{2\pi ik_1t}
\end{matrix}
\right).$$
It is the defining map for $H$ and a group epimorphism. The kernel of $\psi$ is
$$\ker~\psi=\{\Bigl(\left(
\begin{matrix}1&0&k_2e\\
              0&1&0 \\
              0&0&1
\end{matrix}
\right),-e\Bigr)~|~e\in \bbz\}\approx \bbz,$$
and we have an isomorphism $H\approx (\calh^3(\bbz)\times \bbr)/\bbz$. Notice that in $\calh^3(\bbz)\times \bbr$ we have $\ker~\psi\cap \calh^3(\bbz)={\rm id}$, so $\calh^3(\bbz)$ is a subgroup of $H$. In fact, it is a normal subgroup of $H$, and $H/\calh^3(\bbz)\approx\bbr/\bbz\approx S^1$. We now identify $G/H$ with $(G/\calh^3(\bbz))/(H/\calh^3(\bbz))$ and the latter with the join $\caln^3\star_{k_1,k_2} S^3$. First we have
$$G/\calh^3(\bbz)=\bigl(\calh^3(\bbr)\times SU(2)\bigr)/\calh^3(\bbz)=\caln^3\times S^3.$$ 
Consider the action of the $\bbr$ subgroup of $H$ on $\caln^3\times S^3$ given in coordinates $([x,y,z];z_1,z_2)$ by 
$$([x,y,z];z_1,z_2)\mapsto ([x,y,z+k_2t];e^{-2\pi ik_1t}z_1,e^{-2\pi ik_1t}z_2)$$
where $t\in\bbr$, and the brackets denote the equivalence class in $\calh^3(\bbr)$ modulo $\calh^3(\bbz)$. Of course, this action is not effective, since the subgroup $\bbz$ obtained by restricting $t$ to $\bbz$ fixes all points of $\caln^3\times S^3$. However, since $(k_1,k_2)$ are relatively prime, the action of the quotient group $\bbr/\bbz\approx S^1(k_1,k_2)$ is free, and we have
$$G/H\approx (G/\calh^3(\bbz))/S^1(k_1,k_2)\approx (\caln^3\times S^3)/S^1(k_1,k_2).$$
But the right hand side is just the join construction as described in Section 7.6.2 of \cite{BG05}. Thus, $G/H\approx \caln^3\star_{k_1,k_2} S^3$.
\end{proof}

Next we determine the weak homotopy type of $\caln^3\star_{k_1,k_2} S^3$.

\begin{lemma}\label{pi1join}
For each pair of relatively prime positive integers $(k_1,k_2)$ we have
\begin{enumerate}
\item $\pi_1(\caln^3\star_{k_1,1} S^3)\approx\pi_1(T^2\times S^3)\approx \bbz^2.$
\item If $k_2>1$, $\pi_1(\caln^3\star_{k_1,k_2} S^3)$ is a non-Abelian central $\bbz_{k_2}$-extension of $\bbz^2$.
\item $\pi_i(\caln^3\star_{k_1,k_2} S^3)\approx \pi_i(S^3)~ \text{for}~i\geq 2;$ In particular, 
\newline $\pi_2(\caln^3\star_{k_1,k_2} S^3)=0$.
\end{enumerate}
\end{lemma}

\begin{proof}
Applying the long exact homotopy sequence to the bundle 
$$H\ra{1.4}G\ra{1.4}G/H\approx \caln^3\star_{k_1,k_2} S^3$$ 
of Lemma \ref{joinlem} we have
$$\ra{1.7}\pi_i(G)\ra{1.7}\pi_i(G/H)\ra{1.7}\pi_{i-1}(H)\ra{1.7}\pi_{i-1}(G)\ra{1.7} \pi_{i-1}(G/H)\ra{1.7}. $$
Now $G$ is $2$-connected, so we have the group isomorphism $\pi_2(G/H)\approx\pi_{1}(H)$,  and the set bijection\footnote{Generally, $\pi_0(X)$ is just the set of path components of $X$ and has no group structure; however, if $\grG$ is a discrete group $\pi_0(\grG)$ is isomorphic to $\grG$ itself with its group structure.} $\pi_1(G/H)\approx\pi_0(H)$. To proceed further we notice that the connected component $H_0$ of $H$ is the normal subgroup given by matrices of the form
$$\left(
\begin{matrix}1&0&k_2t\\
              0&1&0 \\
              0&0&1
\end{matrix}
\right)\times 
\left(
\begin{matrix}e^{-2\pi ik_1t}&0 \\
              0& e^{2\pi ik_1t}
\end{matrix}
\right),\qquad t\in \bbr.$$
For $k_2\in \bbz^+$ this has the homotopy type of $\bbr$, so $\pi_i(H)=\pi_i(H_0)=0$ for $i\geq 1$. Thus, in particular using Lemma \ref{joinlem} we have 
$$\pi_2(\caln^3\star_{k_1,k_2} S^3)=\pi_2(G/H)\approx\pi_{1}(H)= 0.$$
Then the long exact sequence of the bundle $S^1\ra{1.5}\caln\times S^3\ra{1.5}\caln^3\star_{k_1,k_2} S^3$ gives the short exact sequence
$$0\ra{1.8}\bbz\fract{\grd}{\ra{1.8}}\calh^3(\bbz)\ra{1.8}\pi_1(\caln^3\star_{k_1,k_2} S^3)\ra{1.8} 0,$$
where the connecting homomorphism $\grd$ is given by 
\begin{equation}\label{conhomo}
\grd(n)=\left(
\begin{matrix}1&0&k_2n\\
              0&1&0 \\
              0&0&1
\end{matrix}
\right).
\end{equation}
This gives the group isomorphism 
\begin{equation}\label{pi1centext}
\pi_1(\caln^3\star_{k_1,k_2} S^3)\approx\calh^3(\bbz)/k_2Z(\bbz)
\end{equation}
where $Z(\bbz)$ is the central subgroup of $\calh^3(\bbz)$ consisting of matrices of the form
$$\left(
\begin{matrix}1&0&c\\
              0&1&0 \\
              0&0&1
\end{matrix}
\right),\quad c\in \bbz,$$
which proves (1). Then Equation (\ref{pi1centext}) and the well known homomorphism theorems give the exact sequence
\begin{equation}\label{k2exact}
0\ra{1.8}\bbz_{k_2}\ra{1.8}\pi_1(M^5_{k_1,k_2})\ra{1.8}\bbz^2\ra{1.8}0.
\end{equation}
Moreover, from the structure of the Heisenberg group this is a non-Abelian central extension which proves item (2). The fact that $\pi_i(H)$ vanishes for $i\geq 1$ together with the fibration $H\ra{1.5}G\ra{1.5}G/H$ implies $\pi_i(\caln^3\star_{k_1,k_2} S^3)=\pi_i(G/H)\approx \pi_i(G)\approx\pi_i(S^3)$ for $i\geq 2$ which finishes the proof of the lemma.
\end{proof}

\begin{proof}[Proof of Theorem]
First consider the case $k_2=1$. We make use of a topological rigidity result of Kreck and L\"uck \cite{KrLu09}. Notice that the torus $T^2$ is the classifying space $B\bbz^2$. So we consider the classifying map $f:M^5_{k_1,1}\ra{1.5} B\bbz^2\approx T^2$ to be the composition $M^5_{k_1,1}\ra{1.5}T^2\times S^2\ra{1.5} T^2$. Now by Lemma \ref{pi1join} $M^5_{k_1,1}$ satisfies $\pi_1(M^5_{k_1,1})=\bbz^2$ and $\pi_2(M^5_{k_1,1})=0$, and hence, the hypothesis of Problem 0.16 of \cite{KrLu09} is satisfied, namely, that $\pi_1(M^5_{k_1,1})$ is non-trivial and isomorphic to the fundamental group of a manifold of dimension $\leq 2$, and $\pi_2(M^5_{k_1,1})=0$. Thus, by Theorem 0.18 of \cite{KrLu09} the oriented homeomorphism type, in fact since homeomorphism implies diffeomorphism in dimension five, the oriented diffeomorphism type of $M^5_{k_1,1}$ is determined completely by its second Stiefel-Whitney class $w_2(M^5_{k_1,1})$. More explicitly, $M^5_{k_1,1}$ is an $S^3$-bundle over $T^2$, and there are precisely two such bundles, the trivial one with $w_2(M^5_{k_1,1})=0$, and the non-trivial one with $w_2(M^5_{k_1,1})\neq 0$. 

There are two ways to determine which of the two bundles occurs. One can compute $w_2(M^5_{k_1,k_2})$ explicitly using the fact that it is the mod 2 reduction of the first Chern class $c_1(\cald_{k_1,k_2})$ \cite{BG05}, and the latter is calculated from the pullback of the first Chern class of the quotient $T^2\times S^2$ via transgression. Since we need the Chern class to distinguish contact structures, we give this computation in Lemma \ref{Chernclass} below where we see that it is always an even multiple of a generator of $H^2(M^5_{k_1,k_2},\bbz)/({\rm torsion})$.
Alternatively, Gorbacevi\v{c} \cite{Gor78} has classified 5-dimensional compact homogeneous manifolds. Since our $M^5_{k_1,1}$ is homogeneous and the trivial bundle $T^2\times S^3$ is also, whereas, the non-trivial bundle does not appear on Gorbacevi\v{c}'s list, $M^5_{k_1,1}$ must be the former.

For the general case we can apply Proposition 7.6.7 of \cite{BG05} which gives $M^5_{k_1,k_2}$ as a bundle over $T^2$ with fiber the lens space $L(k_2,1)$. Now (2) of Lemma \ref{pi1join} implies that that the lens  space bundle is non-trivial. Moreover, using Equation (\ref{pi1centext}), an easy computation shows that the commutator group $[\pi_1,\pi_1]$ is $\bbz_{k_2}$, so $H_1(M^5_{k_1,k_2},\bbz)\approx \bbz^2$.
\end{proof}

\begin{lemma}\label{Chernclass}
On the contact manifold $M^5_{k_1,k_2}$ we have $c_1(\cald_{k_1,k_2})=2k_1\grg$ where $\grg$ is a generator of $H^2(M^5_{k_1,k_2},\bbz)/({\rm torsion})$. 
\end{lemma}

\begin{proof}
The join construction gives a circle bundle $\pi:M^5_{k_1,k_2}\ra{1.6}T^2\times S^2$. Choosing a basis $(\gra,\grb)$ for $H^2(T^2\times S^2,\bbz)\approx H^2(T^2,\bbz)\oplus H^2(S^2,\bbz)$, we have $c_1(T^2\times S^2)=2\grb$ and the Euler class of the circle bundle is $k_1\gra+k_2\grb$. The pullback of this class to $M^5_{k_1,k_2}$ vanishes, so $\pi^*\grb=-\frac{k_1}{k_2}\pi^*\gra$. But $\pi^*\grb$ is an integral class, so we must have $\pi^*\gra=-k_2\grg$ for an element $\grg\in H^2(M^5_{k_1,k_2},\bbz)/({\rm torsion})$. Thus,
$$c_1(\cald_{k_1,k_2})=\pi^* c_1(T^2\times S^2)=2\pi^*\grb=2k_1\grg.$$
Furthermore, $\grg$ is a generator of  $H^2(M^5_{k_1,k_2},\bbz)/({\rm torsion})\approx \bbz$ since $\gra$ and $\grb$ are generators of $H^2(T^2\times S^2,\bbz)$.

\end{proof}

\section{Complex surfaces diffeomorphic to $T^2 \times S^2$}\label{comsurf}

Let $(M,J)$ be a complex surface such that $M$ is diffeomorphic to \newline
$T^2 \times S^2$. Then it follows from Atiyah \cite{Ati55,Ati57} and Suwa \cite{Suw69} that $(M,J)$ is a ruled surface of genus $1$;  $(M,J) = {\mathbb P}(E) \rightarrow T^2$, where $T^2$ is equipped with a complex structure $J_\tau, \tau \in {\mathcal M}$ and, without loss of generality, $E \rightarrow T^2$ is a holomorphic rank $2$ vector bundle over the Riemann surface $(T^2,J_\tau)$ of one of the following types
\begin{enumerate}
\item $E$ is a non-split extension 
$$0 \rightarrow {\mathcal O} \rightarrow E \rightarrow  {\mathcal O} \rightarrow 0$$
\item $E= {\mathcal O} \oplus L$, where $L$ is a degree $0$ holomorphic line bundle on $T^2$ and ${\mathcal O}$ denotes the trivial (holomorphic) line bundle on $T^2$.
\item $E= {\mathcal O} \oplus L$, where $L$ is a  holomorphic line bundle on $T^2$ of positive even degree $n$.
\end{enumerate}
Assume the complex structure $J_\tau$ on $T^2$ is fixed. From \cite{Suw69} we have the following
statements, up to biholomorphism.
The  first type is unique  and we denote the ruled surface by $A_{0,\tau}$. 

The family of ruled surfaces of the second type is denoted by  ${\mathcal S}_{0,\tau}$ and is parametrized by ${\mathbb C}\bbp^1$.  It can be described as follows \cite{Fuj92}: let $\grr:\pi_1(T^2)=\bbz^2\ra{1.6} PSU(2)\approx SO(3)$ be a projective unitary representation. Consider the action of $\pi_1(T^2)$ on $\bbc\times \bbc\bbp^1$ given by the covering space action on the first factor and by $\grr$ on the second. We denote the quotient ruled surface by $T^2\times_\grr \bbc\bbp^1 \in   {\mathcal S}_{0,\tau}$. Note that in our case $\grr$ is a homomorphism from an Abelian group to $SO(3)$, so the image $\grr(\bbz^2)$ in $SO(3)$ is Abelian, and so is contained in a maximal Abelian subgroup of $SO(3)$. Since $T^2\times S^2$ is spin the homomorphism $\grr$ lifts to a homomorphism to $SU(2)$. But in $SU(2)$ any Abelian subgroup must be a subgroup of a circle\footnote{$SO(3)$ does have a maximal Abelian subgroup that is not a circle, namely, the Klein four group of diagonal elements. In this case the underlying manifold is the non-trivial $S^2$ bundle over $T^2$ which is not spin (cf. Exercise 6.14 of \cite{McDSa}), and the complex ruled surface is denoted by $A_1$ in \cite{Suw69}.}. It follows that $\grr(\bbz^2)$ is a subgroup (not necessarily closed) of a circle $S^1$. Generally, we shall denote this complex structure as $J_\grr$. The product complex structure corresponds to $\grr$ mapping $\bbz^2$ to the identity, that is, $\pi_1(T^2)=\ker~\grr$. Naturally this is the case where $L={\mathcal O}$ for type (2) above.  We denote the product structure by $S_{0,\tau}$. 

There is exactly one ruled surface of type (3) for each $n \in 2 {\mathbb Z}^+$ and fixed $\grt\in\calm$. We denote this by  $S_{n,\tau}$.

\subsection{Complex analytic families of complex structures}\label{complex structures}
In this paper we are interested in families of complex structures on $T^2\times S^2$ and therefore
we now summarize the main conclusions of Section 3  of \cite{Suw69}:
Let $(M,J)$ be a ruled surface as above of type (1), (2), or (3) and let $\Theta$ denote the sheaf over $(M,J)$ of germs of holomorphic vector fields. Then we have
$dimH^2((M,J),\Theta) = 0$ and 
\begin{enumerate}
\item if $(M,J) = A_{0,\tau}$, \newline 
then
$dimH^0((M,J),\Theta) =dimH^1((M,J),\Theta) = 2$,

\item  if $(M,J) = S_{0,\tau}$, \newline 
then $dimH^0((M,J),\Theta) =dimH^1((M,J),\Theta) = 4$,

\item  if $(M,J) \neq S_{0,\tau}$ and $(M,J) \in {\mathcal S}_{0,\tau}$, \newline 
then $dimH^0((M,J),\Theta) =dimH^1((M,J),\Theta) = 2$,

 \item  if $(M,J) = S_{n,\tau}$, \newline 
then $dimH^0((M,J),\Theta) =dimH^1((M,J),\Theta) = n+1$.
\end{enumerate}
From \cite{KNS58} and \cite{KoSp58b} it then follows that in each case there is a local complex analytic family ${\mathcal J}$ of complex structures on $T^2\times S^2$ such that $J \in {\mathcal J}$ and ${\mathcal J}$ is parametrized by a complex parameter space of dimension equal to $dimH^1((M,J),\Theta)$. In fact, Suwa explicitly constructs effectively parametrized and complete families at $J$.
One of the deformation directions  in each of the cases above corresponds to changing
 $J_\tau$ on the base $T^2$.  
In case (2), two of the deformation directions leads to $A_{0,\tau}$, while the last yields the complex analytic family  $ {\mathcal S}_{0,\tau}$. Case (3) is similar to case (2) without the two deformation directions leading to $A_0$. Finally, case (4) has two deformation directions jumping to $A_{0,\tau}$ and $n-2$ deformation directions jumping to $S_{n-2(k-1),\tau}$ for $k=3,...,n$, where $S_{-m,\tau} = S_{m,\tau}$. Unless, $n=2$, the latter yield, possibly with some double counting, the ruled surfaces $S_{n,\tau}, S_{n-2,\tau},...,S_{0,\tau}$ (the first one corresponding to no deformation).

Consider now the symplectic $2$-form 
$$\omega_{k_1,k_2} = k_1 \omega_1 + k_2 \omega_2$$
on  $T^2 \times S^2$, where $\omega_1$ and $\omega_2$ are the standard area measures on $T^2$ and $S^2$, respectively. Let $\alpha_{k_1,k_2} \in H^2(M,{\mathbb R})$ denote the cohomology class of $\omega_{k_1,k_2}$. 

\begin{lemma}\label{Kahclass}
For any $(M,J) \in  {\mathcal S}_{0,\tau} \cup \{ A_{0,\tau} \}$,
$\alpha_{k_1,k_2}$ is a K\"ahler class if and only if $k_1,k_2 >0$.
For any $(M,J) = S_{n,\tau}$, $n \in 2 {\mathbb Z}^+$, $\alpha_{k_1,k_2}$ is a K\"ahler class if and only if
$k_2>0$ and $\frac{k_1}{k_2} > n/2$.
\end{lemma}

\begin{proof}
In either case the zero  section, $E_{n}$ of $M \rightarrow 
T^2$ has the property that $E_{n}\cdot E_n = n$ where $n=0$ if $(M,J) \in 
 {\mathcal S}_{0,\tau}\cup \{ A_{0,\tau} \}$.
(If $(M,J) = S_{0,\tau}$, $E_0= T^2\times \{pt \}$.)
If $C$ denotes a fiber of the ruling $M \rightarrow 
T^2$, then $C\cdot C=0$, while $C \cdot E_{n} =1$.
Any real cohomology class in the two dimensional space 
$H^{2}(M, {\mathbb R})$ may be written as a linear combination of 
(the Poincare 
duals of) $E_{n}$ and $C$, 
$$
m_{1} E_{n} +m_{2} C\, .
$$
In particular, the K\"ahler cone ${\mathcal K}$ corresponds to $ m_{1} > 0, 
m_{2} > 0$ (see \cite{Fuj92} or Lemma 1 in \cite{To-Fr98}). By integrating 
$\alpha_{k_1,k_2}$ over $E_0$ and $C$ we easily get that
${\rm PD}(\alpha_{k_1,k_2}) = k_2 E_0 + k_1 C$ where PD means Poincar\'e dual, and since $E_0 = E_n - \frac{n}{2} C$ the lemma now follows. 
\end{proof}

As a consequence of this Lemma, if we start with a ruled surface $(M,J)$ diffeomorphic to $T^2 \times S^2$ such that $\alpha_{k_1,k_2}$ is a K\"ahler class, then  $\alpha_{k_1,k_2}$ remains a 
K\"ahler class for all the complex structures arising from the deformation families above.

\section{Existence of Extremal K\"ahler metrics}\label{exKmet}
{\em Extremal K\"ahler metrics} are generalizations of constant
scalar curvature K\"ahler metrics:
Let $(M,J)$ be a compact complex manifold admitting at least one
K\"ahler metric. For a
particular K\"ahler class $\alpha$, let $\alpha^+$ denote the 
set of all K\"ahler forms in $\alpha$.

Calabi \cite{Cal82} suggested that one should look for extrema of the 
following functional $\Phi$ on $\alpha^+$:
\[
\Phi : \alpha^+ \rightarrow {\mathbb R}
\]
\[
\Phi(\omega) = \int_M s^2 d\mu,
\]
where $s$ is the scalar curvature and $d\mu$ is the volume form of the
K\"ahler metric corresponding to the K\"ahler form $\omega$.
Thus $\Phi$ is the square of the $L^2$-norm of the scalar curvature.

\bigskip

\begin{proposition}\cite{Cal82}
The K\"ahler form $\omega \in \alpha^+$ is an
extremal point of $\Phi$ if and only
if the gradient vector field $grad \, s$ is a holomorphic real vector field, that is
$\pounds_{grad \, s}J=0$. When this happens the metric $g$ corresponding to $\omega$ is
called an {\em extremal K\"ahler metric}.
\end{proposition}
Notice that if $\pounds_{grad \, s}J=0$, then $Jgrad\, s$ is a Hamiltonian Killing vector field inducing Hamiltonian isometries.

As follows from Fujiki \cite{Fuj},
the complex surface $A_{0,\tau}$ does not admit any extremal metric at all\footnote{Notice that in the first paragraph of the proof of Theorem 4.6 in  \cite{ACGT08b}, it is inadvertently and incorrectly implied that $E$ of $A_{0,\tau}$ is polystable (and hence $A_{0,\grt}$ should admit a CSC K\"ahler metric). This is obviously not true. In fact, what should have been said is that \underline{the other} of the the two possible cases of $ {\mathbb P}(E) \rightarrow T^2$ with $E$ indecomposable has $E$ polystable. That case has also $E$ non-spin and so the bracket comment in Theorem 4.6 of \cite{ACGT08b} is not true and should be ignored.}. 

\begin{lemma}\label{A0lem}
The complex surface $T^2\times S^2$ with a non-split complex structure $A_{0,\grt}$ has no non-trivial Hamiltonian Killing vector fields and no extremal K\"ahler metrics.
\end{lemma}

\begin{proof}
To see that $A_{0,\tau}$ admits no non-trivial Hamiltonian Killing vector fields with respect to any K\"ahler metric,
assume that we did have such a vector field $X$. Then, since $A_{0,\tau}$ is compact, $X$ would have to vanish  somewhere and thus if $\pi$ denotes the projection of $A_{0,\tau}$ to $T^2$, $\pi_* X$ would be a holomorphic vector field on $T^2$ with a zero. It is well known that holomorphic vector fields on $T^2$ either vanish everywhere or nowhere. Thus $\pi_* X =0$. This means that $X$ would induce a group of fiber preserving automorphisms of  $A_{0,\tau}$. By compactness of $A_{0,\tau}$ any Hamiltonian Killing vector field $X\neq0$ induces a group of automorphisms whose closure is  $S^1$ or $T^2$. By the fiber preservation, the latter possibility is clearly not possible and a fiber preserving $S^1$ action would cause $E$ to split into two holomorphic line bundles (cf. \cite{ACGT11} Lemma 1). So there are no non-trivial Hamiltonian Killing vector fields.

The lack of  non-trivial Hamiltonian Killing vector fields implies that any extremal K\"ahler metric would have to have constant scalar curvature. However, by the well known Lichn\'erowicz-Matsushima Theorem (cf. \cite{Gau09b}) the Lie algebra of holomorphic vector fields on $(M,J)$ must be a reductive complex Lie algebra. For $(M,J)=  A_{0,\tau}$ this is not the case \cite{Mar71} and so the manifold admits no extremal K\"ahler metrics.
\end{proof}

 If $(M,J) \in {\mathcal S}_{0,\tau}$, then there is a constant scalar curvature (CSC) K\"ahler metric in each K\"ahler class of the K\"ahler cone on $T^2\times_\grr \bbc\bbp^1$. These are called {\it quasi-stable} in \cite{Fuj92}. When $\grr$ is the identity  $(M,J) = T^2 \times {\mathbb C}\bbp^1$ is simply a product of constant curvature K\"ahler metrics on $T^2$ and ${\mathbb C}\bbp^1$ respectively. In general,  $T^2\times_\grr \bbc\bbp^1$ is a flat ${\mathbb C}P^1$-bundle and so the local products of CSC K\"ahler metrics inherited from product CSC K\"ahler metrics on the universal cover $\bbc\times \bbc\bbp^1$ exhaust the K\"ahler cone.

If $(M,J) = S_{n,\tau}$, $n \in 2 {\mathbb Z}^+$, then there is an extremal K\"ahler metric (non-CSC) in every K\"ahler class \cite{Hwa94} (see also \cite{To-Fr02}) arising from a Calabi type construction:
Recall that $M = {\mathbb P}({\mathcal O} \oplus L) \rightarrow T^2$,
where $L$ is a  
holomorphic line bundle of degree $n\in 2{\mathbb Z}^+$ on $T^2$, and 
${\mathcal O}$ is the trivial holomorphic line bundle.
Let $g_{T^2}$ be the
K\"ahler metric
on $T^2$ of constant zero scalar curvature, with K\"ahler form
$\omega_{T^2}$, such that
$c_{1}(L) = [\frac{\omega_{T^2}}{2 \pi}]$.

The natural $\mathbb{C}^*$-action on $L$ 
extends to a holomorphic
$\mathbb{C}^*$-action on $M$. The open and dense set $M_0$ of stable points with 
respect to the
latter action has the structure of a principal $\mathbb{C}^*$-bundle over the stable 
quotient.
The hermitian norm on the fibers induces via a Legendre transform a function
${\gz}:M_0\rightarrow (-1,1)$ whose extension to $M$ consists of the critical manifolds
$E_{n}:={\gz}^{-1}(1)=P({\mathcal O} \oplus 0)$ and 
$E_{\infty}:= {\gz}^{-1}(-1)=P(0 \oplus L)$.
To build the so-called admissible metrics \cite{ACGT08} on $M$ we 
proceed as follows. Let  $\theta$ be a connection one form for the 
Hermitian metric on $M_0$, with curvature
$d\theta = \omega_{T^2}$. Let $\Theta$ be a smooth real function with 
domain containing
$(-1,1)$. Let $r$ be a real number such that $0 < r < 1$.
Then an admissible K\"ahler metric
is given on $M_0$ by
\begin{equation}\label{metric}
g  =  \frac{1+r {\gz}}{r} g_{T^2}
+\frac {d{\gz}^2}
{\Theta ({\gz})}+\Theta ({\gz})\theta^2\,
\end{equation}
with K\"ahler form 
\begin{equation}
\omega =  \frac{1+r{\gz}}{r}\omega_{T^2}
+d{\gz}\wedge \theta\,. \label{kf}
\end{equation}
The complex structure yielding this
K\"ahler structure is given by the pullback of the base complex structure
along with the requirement 
\begin{equation}\label{complex}
Jd{\gz} = \Theta \theta.
\end{equation}
The function ${\gz}$ is 
Hamiltonian
with $K= J grad_g{\gz}$ a Killing vector field. Observe that $K$ 
generates the circle action which induces the holomorphic
$\mathbb{C}^*$- action on $M$ as introduced above.
In fact, ${\gz}$ is the moment 
map on $M$ for the circle action, decomposing $M$ into 
the free orbits $M_{0} = {\gz}^{-1}((-1,1))$ and the special orbits 
${\gz}^{-1}(\pm 1)$. Finally, $\theta$ satisfies
$\theta(K)=1$.
In order that $g$ (be a genuine metric and) extend to all of $M$,
$\Theta$ must satisfy the positivity and boundary
conditions
\begin{align}
\label{positivity}
(i)\ \Theta({\gz}) > 0, \quad -1 < {\gz} <1,\quad
(ii)\ \Theta(\pm 1) = 0,\quad
(iii)\ \Theta'(\pm 1) = \mp 2.
\end{align}
The last two of these are together necessary and sufficient for
the smooth compactification of $g$. 

Note that in the above set-up different choices of $\Theta$  determines different compatible complex structures
$J$ with the same fixed symplectic form $\omega$ as the K\"ahler form. However, for each 
$\Theta$ there is an $S^1$-equivariant diffeomorphism pulling back $J$ to the original fixed complex structure of  $S_{n,\tau}$ in such a way that the K\"ahler form of the new K\"ahler metric is in the same cohomology class as
$\omega$. Therefore, with all else fixed, we may view the set of the functions $\Theta$ satisfying \eqref{positivity} as parametrizing a family of K\"ahler metrics within the same K\"ahler class of 
$S_{n,\tau}$ \cite{ACGT08}.

It is easy to see that the K\"ahler class of a metric as in \eqref{metric} is given by
\[
{\rm PD}([\omega]) = 4\pi E_{n}+ \frac{2\pi(1-r)n}{r} C.
\]
From the proof of Lemma \ref{Kahclass}, we see that, up to rescaling, the set $\{0<r<1\}$ exhausts the entire K\"ahler cone.
Finally, one may check by direct calculation \cite{ACGT08} that $g$ as in \eqref{metric} is extremal if and only if
\[
\Theta(\gz) = \frac{(1-\gz^{2})(2r^{2}\gz^{2} + r (6-2r^{2})\gz + (6 - 4r^{2}))}{(1+r \gz)2(3-r^{2})}.
\]
For any choice of $0<r<1$ this is a function satisfying all conditions in \eqref{positivity} and thus any K\"ahler class admits an extremal K\"ahler metric. None of these extremal K\"ahler metrics have constant scalar curvature.

\subsection{Families of complex structures with extremal metrics in  $\alpha_{k_1,k_2}$ } 

Now, if we start with a ruled surface $(M,J)$ diffeomorphic to $T^2 \times S^2$ such that $\alpha_{k_1,k_2}$ is a K\"ahler class admitting an extremal K\"ahler metric, then we know that either
$(M,J) \in {\mathcal S}_{0,\tau}$, and then the extremal metric is a CSC K\"ahler metric, or 
$(M,J) = S_{n,\tau}, n \in 2{\mathbb Z}^+$. From Section \ref{complex structures} together with the above observations we see that in the first case we have a two dimensional complex parameter family of complex structures such that $\alpha_{k_1,k_2}$ remains a K\"ahler class admitting a CSC K\"ahler metric. The second case contains two subcases; if $n=2$ we have a one dimensional complex parameter family (corresponding to changing the complex structure on the base) of complex structures such that $\alpha_{k_1,k_2}$ remains a K\"ahler class admitting an
extremal K\"ahler metric (all non-CSC), if $n=4,6,...,$ we have an $(n-1)$-dimensional complex family of complex structures such that $\alpha_{k_1,k_2}$ remains a K\"ahler class admitting an
extremal K\"ahler metric. A one dimensional sub-parameter family contains complex structures (all biholomorphic to $S_{0,\tau}$) admitting  $CSC$ K\"ahler metrics whereas the rest are non-CSC.

\begin{remark}
Suppose $\alpha_{k_1,k_2}$ is a K\"ahler class for a given complex structure $J$ on $T^2 \times S^2$. According to McDuff \cite{McD94}, up to isotopy, there is only one symplectic form in the class $\alpha_{k_1,k_2}$. 
In particular, if $\alpha_{k_1,k_2}$ admits some extremal K\"ahler metric $g$ w.r.t. $J$ 
 with K\"ahler form $\omega$, then there exists a diffeomorphism $\phi$ such that $\phi^*\omega = \omega_{k_1,k_2}$. Then, $\phi^*J$ is a complex structure compatible with $\omega_{k_1,k_2}$ such that $\omega_{k_1,k_2}$ is the K\"ahler form of an extremal K\"ahler metric.
 \end{remark}

\section{Hamiltonian Circle Actions on $T^2\times S^2$}\label{hamS1}

The purpose of this section is to show that the Hamiltonian circle actions corresponding to the complex structures $S_{2m,\grt}$ (including $m=0$) discussed in Sections \ref{comsurf} and \ref{exKmet} belong to distinct conjugacy classes of maximal tori in the group $\gH\ga\gm(T^2\times S^2,\gro_{k_1,k_2})$ for $m=0,\cdots, \lceil\frac{k_1}{k_2}\rceil-1$.

Fix a symplectic form $\gro_{k_1,k_2}=k_1\gro_1+k_2\gro_2$ on $T^2\times S^2$ with $k_1,k_2\in\bbz^+$ relatively prime, and let $\gS\gy\gm(T^2\times S^2,\gro_{k_1,k_2})$ denote its group of symplectomorphisms. It is a Fr\'echet Lie group locally modelled on its Lie algebra 
$$\gs\gy\gm(T^2\times S^2,\gro_{k_1,k_2})=\{X\in\gX(T^2\times S^2)~|~\pounds_X\gro_{k_1,k_2}=0\},$$ 
where $\gX(M)$ denotes the Lie algebra of smooth vector fields on $M$. We are interested in the ideal of Hamiltonian Killing vector fields $\gh\ga\gm(T^2\times S^2,\gro_{k_1,k_2})$ of $\gs\gy\gm(T^2\times S^2,\gro_{k_1,k_2})$ consisting of those vector fields $X$ such that the 1-form $X\hook \gro_{k_1,k_2}$ is exact. The normal subgroup of {\it Hamiltonian isotopies} $\gH\ga\gm(T^2\times S^2,\gro_{k_1,k_2})$ is defined to be the subgroup of  $\gS\gy\gm(T^2\times S^2,\gro_{k_1,k_2})$ generated by smooth families of Hamiltonian Killing vector fields connected to the identity.

Consider the symplectic 4-manifolds $(T^2\times S^2,\gro_{k_1,k_2})$ together with the diffeomorphisms $\varphi_{2m}:T^2\times S^2\ra{1.6} \bbp(\calo\oplus L)$ where $L$ is a line bundle on $T^2$ of degree $2m$. Transport the complex structure on $ \bbp(\calo\oplus L)$ to $T^2\times S^2$ via $\varphi_{2m}$. Let $J_{2m}$ denote this complex structure on $T^2\times S^2$. It is compatible with the symplectic form, and it follows from Lemma \ref{Kahclass} that $(T^2\times S^2,\gro_{k_1,k_2},J_{2m})$ with $m\in \bbz^+$ is K\"ahler if and only if $k_1>mk_2$.

Hamiltonian $S^1$ actions on 4-manifolds were first studied independently by Ahara-Hattori \cite{AhHa91} and Audin \cite{Aud90}. Later Karshon \cite{Kar99} classified the Hamiltonian circle actions on 4-manifolds in terms of certain labelled  graphs. These graphs are determined by the fixed point set of the $S^1$ action. See also Chapter VIII of \cite{Aud04}. 

We write a point of the total space $W$ of the projective bundle $\pi:\bbp(\calo\oplus L)\ra{1.6} T^2$ as  $(w,[u,v])$ where $[u,v]$ are homogeneous coordinates in the $\bbc\bbp^1$ fiber $\bbp(\calo\oplus L_w)$ at $w\in T^2$. Define the circle action on $W$ by $\tilde{\cala}(\grl):W\ra{1.6} W$ by $\tilde{\cala}(\grl)(w,[u,v])=(w,[u,\grl v])$ where $\grl\in \bbc$ with $|\grl|=1$. This action is clearly holomorphic. Let $\cala_{2m}(\grl) =\varphi^{-1}_{2m}\circ\tilde{\cala}(\grl)\circ\varphi_{2m}$ denote the transported action on $T^2\times S^2$. It is holomorphic with respect to $J_{2m}$. The fixed point set of the action $\tilde{\cala}(\grl)$ is the disjoint union of sections $E_\infty=\gz^{-1}(-1)=(w,[0,v])$ and $E_n=\gz^{-1}(1)=(w,[1,0])$. Then we have \cite{Aud90}

\begin{lemma}\label{fixHam}
For each $n\in 2\bbz_{\geq 0}$ satisfying $n<\frac{2k_1}{k_2}$ and $\grl\in S^1$ we have
$\cala_{n}(\grl)\in \gH\ga\gm(T^2\times S^2,\gro_{k_1,k_2}) $.
\end{lemma}

Thus, for each $n=2m\in 2\bbz^+$ satisfying $m<\frac{k_1}{k_2}$ we have a monomorphism   
$$\cala_n:S^1\ra{1.6} \gH\ga\gm(T^2\times S^2,\gro_{k_1,k_2})\subset \gS\gy\gm_0(T^2\times S^2,\gro_{k_1,k_2}),$$
where the subscript $0$ on a group denotes its connected component. We simplify our notation following \cite{McD01,Bus10} to some extent and define
\begin{equation}\label{sympgrp}
G_{k_1,k_2}=\gS\gy\gm(T^2\times S^2,\gro_{k_1,k_2})\cap \gD\gi\gf\gf_0(T^2\times S^2).
\end{equation}
Clearly, $\gS\gy\gm_0(T^2\times S^2,\gro_{k_1,k_2})\subset G_{k_1,k_2}$. We claim that the circle actions $\cala_n$ belong to different conjugacy classes of maximal tori in $\gH\ga\gm(T^2\times S^2,\gro_{k_1,k_2})$ where conjugacy is taken under the larger group $G_{k_1,k_2}$. In order to see this we employ the work of Bu{\c{s}e \cite{Bus10} and consider the rational homotopy group $\pi_1(G_{k_1,k_2})\otimes \bbq$. Note that tensoring with $\bbq$ is defined here since the fundamental group of any topological group is Abelian. Recall the {\it ceiling function} $\lceil a\rceil$ defined to be the smallest integer greater than or equal to $a$. Before stating the main result of this section, we recall that $\gH\ga\gm(T^2\times S^2,\gro_{k_1,k_2})$ cannot contain a torus of dimension greater than one.

\begin{theorem}\label{conjmaxtori}
There are exactly $\lceil\frac{k_1}{k_2}\rceil$ conjugacy classes (under $G_{k_1,k_2}$) of maximal tori in $\gH\ga\gm(T^2\times S^2,\gro_{k_1,k_2})$. Each of these classes is represented by one of the circle subgroups $\cala_{2m}(S^1)$, $m=0,\ldots,\lceil\frac{k_1}{k_2}\rceil$.
\end{theorem}

\begin{proof}
First notice that since $\gH\ga\gm(T^2\times S^2,\gro_{k_1,k_2})$ is normal in $G_{k_1,k_2}$ it makes sense to consider conjugacy under this larger group. Then by Lemma \ref{fixHam} we have $\lceil\frac{k_1}{k_2}\rceil$ Hamiltonian circle actions given by $\cala_{2m}$ with $m=0,\cdots\lceil\frac{k_1}{k_2}\rceil-1$. According to Lemma 4.3 of \cite{Bus10} the induced maps in rational homotopy $[\cala_{2m}]$ satisfy the equation
\begin{equation}\label{Buseq}
[\cala_{2m}]=m[\cala_{2}]\in \pi_1(G_{k_1,k_2})\otimes \bbq
\end{equation}
for $m=1,\cdots,\lceil\frac{k_1}{k_2}\rceil -1$. Moreover, they are non-trivial in $\pi_1(G_{k_1,k_2})\otimes \bbq$. But elements of the vector space $\pi_1(G_{k_1,k_2})\otimes \bbq$ are invariant under conjugacy giving altogether at least $\lceil\frac{k_1}{k_2}\rceil$ conjugacy classes of circles (maximal tori) in $\gH\ga\gm(T^2\times S^2,\gro_{k_1,k_2})$. 

We now show there are no other such conjugacy classes. Suppose there is another Hamiltonian circle action that does not belong to one of the conjugacy classes described above. By Theorem 7.1 of \cite{Kar99} there is a compatible complex structure $J'$ such that $(T^2\times S^2,\gro_{k_1,k_2},J')$ is K\"ahler. But then it must satisfy the bound of Lemma \ref{Kahclass}, and the classification of complex structures \cite{Ati57,Suw69} on ruled surfaces implies that $J'$ must belong to one on the list in Section \ref{complex structures} for which the conjugacy classes have been determined. Since maximal tori in $\gA\gu\gt(\gro_{k_1,k_2},J')$ are unique up to conjugacy, this gives a contradiction. (See also Chapter VIII of \cite{Aud04}).
\end{proof}

\section{Sasakian Structures on $M^5_{k_1,k_2}$}

Sasakian structures can be easily constructed on $M^5_{k_1,k_2}$ by applying the Inversion Theorem 7.1.6 of \cite{BG05} to the K\"ahlerian structures on ruled surfaces of genus one discussed in Section \ref{complex structures}. Consider the symplectic 4-manifold $(T^2\times S^2,\gro_{k_1,k_2})$ and construct the principal $S^1$-bundle $\pi:M^5_{k_1,k_2}\ra{1.6} T^2\times S^2$ over it corresponding to the class $\gra_{k_1,k_2}=[\gro_{k_1,k_2}]\in H^2(T^2\times S^2,\bbz)$. Let $\eta_{k_1,k_2}$ be a connection 1-form in $M^5_{k_1,k_2}$ satisfying $d\eta_{k_1,k_2}=\pi^*\gro_{k_1,k_2}$. By Boothby-Wang $(M^5_{k_1,k_2},\eta_{k_1,k_2})$ is a regular contact manifold with contact 1-form $\eta_{k_1,k_2}$ and contact bundle $\cald_{k_1,k_2}=\ker\eta_{k_1,k_2}$. Choosing complex structures $J\in {\mathcal S}_{0,\tau}\cup \{A_{0,\grt}\}\cup S_{2m,\grt}$, we obtain K\"ahler structures $(\gro_{k_1,k_2},J,h_{k_1,k_2})$ on $T^2\times S^2$ subject to the conditions that $k_1,k_2$ are relatively prime positive integers and $m<\frac{k_1}{k_2}$ and the K\"ahler metric\footnote{The opposite convention to that usually used in K\"ahlerian geometry is typically used in Sasakian geometry, namely, that $\gro\circ(J\otimes \BOne)$ is positive.} is given by $h_{k_1,k_2}=\gro_{k_1,k_2}\circ(J\otimes\BOne)$. By taking the horizontal lift of $J$ and extending it to a section $\Phi$ of the endomorphism bundle of $TM^5_{k_1,k_2}$ by imposing $\Phi\xi_{k_1,k_2}=0$ where $\xi_{k_1,k_2}$ is the Reeb vector field of $\eta_{k_1,k_2}$, we obtain Sasakian structures $(\xi_{k_1,k_2},\eta_{k_1,k_2},\Phi,g)$ on $M^5_{k_1,k_2}$.

\subsection{Families of Sasakian Structures associated to $\cald_{k_1,k_2}$}\label{famsas}
We easily obtain families of transverse complex structures by lifting the families of complex structures from the base manifold. Nevertheless, it is interesting to see how this relates to applying Kodaira-Spencer deformation theory to the transverse geometry of the characteristic foliation $\calf_\xi$ of a fixed Sasakian structure $\cals=(\xi,\eta,\Phi,g)$. So we can apply Proposition 8.2.6 of \cite{BG05} to our case and use the fact \cite{Suw69} that for any ruled surface $\calz$ we have $H^2(\calz,\Theta)=0$ where $\Theta$ is the sheaf of germs of holomorphic vector fields on $\calz$. If $\Theta_{\calf_\xi}$ denotes the sheaf of germs of transverse holomorphic vector fields on the Sasakian circle bundle $M^5_{k_1,k_2}$ over $\calz$, the aforementioned proposition gives the exact sequence
\begin{equation}\label{transexactseq}
0\ra{2.4}H^1(\calz,\Theta)\ra{2.5}H^1(M^5_{k_1,k_2},\Theta_{\calf_\xi}) \ra{2.5}H^0(\calz,\Theta) \ra{2.4} 0.
\end{equation}
So the transverse holomorphic deformations on $M^5_{k_1,k_2}$ arise in two distinct ways, first from the holomorphic deformations of the base, and second from the holomorphic symmetries of the base. The first inclusion map is the differential of the lift of a complex structure $\hat{J}\in {\mathcal S}_{0,\tau}\cup \{A_{0,\grt}\}\cup S_{2m,\grt}$ to a strictly pseudoconvex CR structure $(\cald_{k_1,k_2},J)$ on $M^5_{k_1,k_2}$. Extending $J$ to the endomorphism $\Phi$ by demanding $\Phi\xi=0$ gives families of Sasakian structures with the same Reeb vector field $\xi$. As mentioned above the inverse to the Boothby-Wang construction (cf. Theorem 7.1.6 of \cite{BG05}) guarantees that these structures are Sasakian with underlying CR structure $(\cald,J)$. In fact they all share the same contact 1-form $\eta$. By abuse of notation we will also use ${\mathcal S}_{0,\tau}\cup \{A_{0,\grt}\}\cup S_{2m,\grt}$ as the local parameter space for the transverse complex structures, writing $J\in {\mathcal S}_{0,\tau}\cup \{A_{0,\grt}\}\cup S_{2m,\grt}$.

The relation with infinitesimal symmetries is more involved. In order that a holomorphic vector field on $\calz$ give Sasakian deformations of Sasakian structures it is necessary that it also be Hamiltonian which means in our case that it be one of the circle actions discussed in Section \ref{hamS1}. For it is precisely the Hamiltonian Killing vector fields $\hat{X}$ that lift to an infinitesimal automorphism of the Sasakian structure by Corollary 8.1.9 of \cite{BG05}. Let us see exactly how a Hamiltonian Killing vector field lifts.
\begin{lemma}\label{liftham}
Let  $M$ be a quasiregular Sasakian manifold with Sasakian structure $\cals=(\xi,\eta,\Phi,g)$ and let $\pi:M\ra{1.6}\calz$ be the orbifold Boothby-Wang map to the K\"ahler orbifold $\calz$ with K\"ahler form $\gro$. Let $\cX$ be a vector field on $\calz$ leaving both the K\"ahler form $\gro$ and the complex structure $J$ invariant. Then $\cX$ lifts to an infinitesimal automorphism $X$ of the Sasakian structure $\cals$ that is unique modulo the ideal $\cali_\xi$ generated by $\xi$ if and only if it is Hamiltonian. Furthermore, if $\cX$ is Hamiltonian with Hamiltonian function $H$ satisfying $\cX\hook\gro=-dH$, then $X$ can be represented by $\cX^h+\pi^*H\xi$ where $\cX^h$ denotes the horizontal lift of $\cX$.
\end{lemma}

\begin{proof}
The first claim is just Corollary 8.1.9 of \cite{BG05}. To see that the lift can be represented by $\cX^h+\pi^*H\xi$ we look for a smooth basic function $a$, which exists by the first part, such that $X=\cX^h+a\xi$ and compute $0=\pounds_X\eta=X\hook d\eta + d(\eta(X))$ implying 
$$da=-\cX^h\hook d\eta=-\cX^h\hook \pi^*\gro=-\pi^*(\cX\hook\gro)=d\pi^*H.$$
So we choose $a=H$.
\end{proof}

The Reeb vector field together with the lift $X=\hat{X}^h+\eta(X)\xi$ span the Lie algebra $\gt_2$ of a maximal torus $\gT^2\in \gC\go\gn(M^5_{k_1,k_2},\eta)$. So we obtain deformed Sasakian structures by choosing another Reeb vector field $\xi'$ representing an element in the Sasaki cone $\grk(\cald,J)$.

\subsection{Bouquets of Sasakian Structures}
If $J\in {\mathcal S}_{0,\tau}$ then the Hamiltonian circle action leaving the K\"ahler structure $(\gro_{k_1,k_2},J,h_{k_1,k_2})$ invariant is $\cala_0$, whereas, if $J\in S_{2m,\grt}$ and $m<\frac{k_1}{k_2}$, the Hamiltonian circle action leaving $(\gro_{k_1,k_2},J,h_{k_1,k_2})$ invariant is $\cala_{2m}$. Now according to \cite{Ler02b,Boy10a} these Hamiltonian circle groups lift to maximal tori of dimension two in the contactomorphism group $\gC\go\gn(M^5_{k_1,k_2},\eta_{k_1,k_2})$. Furthermore, applying Theorem \ref{conjmaxtori} the corresponding circle groups $\cala_{2m'}$ and $\cala_{2m}$ lift to maximal tori  in $\gC\go\gn(M^5_{k_1,k_2},\eta_{k_1,k_2})$ that are non-conjugate in the larger group $\gC\go\gn(M^5_{k_1,k_2},\cald_{k_1,k_2})$ when $0\leq m'< m<\frac{k_1}{k_2}$. Since there are $\lceil \frac{k_1}{k_2}\rceil$ such Hamiltonian circle subgroups, there are $\lceil \frac{k_1}{k_2}\rceil$ maximal tori of dimension two in the contactomorphism group all containing the ray generated by the Reeb vector field $\xi_{k_1,k_2}$. In fact, they intersect precisely in this ray. Thus, using Theorem \ref{conjmaxtori} we have
\begin{proposition}\label{contconj}
The contactomorphism group $\gC\go\gn(M^5_{k_1,k_2},\cald_{k_1,k_2})$ contains at least $\lceil \frac{k_1}{k_2}\rceil$ distinct conjugacy classes of maximal tori of dimension $2$ of Reeb type, and exactly $\lceil \frac{k_1}{k_2}\rceil$ conjugacy classes of maximal tori of dimension $2$ of Reeb type that intersect in the ray of the Reeb vector field $\12$.
\end{proposition}

The second statement of Proposition \ref{contconj} can be reformulated as
\begin{corollary}\label{contconjcor}
The strict contactomorphism group $\gC\go\gn(M^5_{k_1,k_2},\eta_{k_1,k_2})$ contains exactly $\lceil \frac{k_1}{k_2}\rceil$ distinct conjugacy classes of maximal tori of dimension $2$ of Reeb type.
\end{corollary}

As a consequence of this we have

\begin{theorem}\label{sasbouq}
For each positive integer $k_2$ the 5-manifolds $M^5_{k_1,k_2}$ admit a countably infinite number of distinct contact structures $\cald_{k_1,k_2}$ labelled by $k_1\in\bbz^+$ which is relatively prime to $k_2$ each having a Sasaki $N$-bouquet $\gB_N(\cald_{k_1,k_2})$ with $N=\lceil\frac{k_1}{k_2}\rceil$ consisting of 2-dimensional Sasaki cones intersecting in a ray. In particular, the manifold $M^5_{k,1}\approx T^2\times S^3$ admits a countably infinite number of distinct contact structures $\cald_k$ labelled by $k\in\bbz^+$ each having a Sasaki $k$-bouquet of Sasakian structures consisting of 2-dimensional Sasaki cones intersecting in a ray.
\end{theorem}

\begin{proof}
The fact that the contact structures $\cald_{k_1,k_2}$ and $\cald_{k_1',k_2}$ are inequivalent when $k_1'\neq k_1$ follows from Lemma \ref{Chernclass}. The statement about the bouquets is a consequence of the discussion above and Theorem \ref{conjmaxtori}.
\end{proof}

\subsection{The Sasaki Cones}
Here we determine the Sasaki cones associated to the different CR structures on $M^5_{k_1,k_2}$. Consider the Sasaki cone $\grk(\cald_{k_1,k_2},J)$. As discussed at the end of Section \ref{famsas} the circle actions on $M^5_{k_1,k_2}$ are determined by lifting the Hamiltonian circle actions $\cala_{2m}$ for $m=0,\cdots,\lceil\frac{k_1}{k_2}\rceil-1$ on $T^2\times S^2$ horizontally to $M^5_{k_1,k_2}$ and extending it to leave the contact structure invariant. Let $\cals=(\xi,\eta,\Phi,g)$ be a regular Sasakian structure on a compact manifold $M$ fibering over $T^2\times S^2$ with its K\"ahler form $\gro$ and projection map $\pi:M\ra{1.5} T^2\times S^2$. Then according to Lemma \ref{liftham} a Hamiltonian Killing vector field $\cX$ lifts to an element $X\in \ga\gu\gt(\cals)$ giving a circle action on $M^5_{k_1,k_2}$. We call the circle action generated by $X$ on $M$ the {\it induced Hamiltonian circle action} on $M$.

\begin{lemma}\label{sasconeprop}
Consider the Sasakian structure $\cals_{k_1,k_2}=(\12,\eta_{k_1,k_2},\Phi_\grt,g)$ on the 5-manifold $M^5_{k_1,k_2}$ with $\Phi_\grt |_{\cald_{k_1,k_2}}=J_\grt$ for $\grt\in\calm$.
Let $X_{2m}$ denote the infinitesimal generator of the induced Hamiltonian circle action on $M^5_{k_1,k_2}$. 
\begin{enumerate}
\item If 
$$J_\grt\in {\mathcal S}_{0,\grt}\sqcup\bigsqcup_{m=1}^{\lceil\frac{k_1}{k_2}\rceil-1} S_{2m,\grt},$$ 
the Sasaki cone has dimension two and  is determined by 
$$\grk(\cald_{k_1,k_2},J_\grt)=\{a\xi_{k_1,k_2}+bX_{2m}~|~a+k_2b\eta_2(X_{2m})>0\},$$
where $\eta_2$ is the standard contact form on $S^3$. 
\item If $J\in A_{0,\grt}$ the Sasaki cone $\grk(\cald_{k_1,k_2},J)$ has dimension one consisting only of the ray of the Reeb vector field $\12$.
\end{enumerate}
\end{lemma}

\begin{proof} Applying the Boothby-Wang construction to the symplectic manifold $(T^2\times S^2,\gro_{k_1,k_2})$ gives 
$M^5_{k_1,k_2}$ as the total space of a principal $S^1$ bundle over the symplectic manifold $(T^2\times S^2,\gro_{k_1,k_2})$. Moreover, a choice of connection 1-form $\eta_{k_1,k_2}$ in this principal bundle such that $d\eta_{k_1,k_2}=\pi^*\gro_{k_1,k_2}$ where $\pi:M^5_{k_1,k_2}\ra{1.5} T^2\times S^2$ is natural projection defines a contact structure $\cald_{k_1,k_2}=\ker\eta_{k_1,k_2}$ on $M^5_{k_1,k_2}$. Letting $\xi_{k_1,k_2}$ be the fundamental vertical vector field on $M^5_{k_1,k_2}$ corresponding to the element $1\in \bbr$ identified as the Lie algebra of $S^1$ gives the Reeb vector field of $\eta_{k_1,k_2}$. From our construction in Section \ref{secdiff} we have a commutative diagram
\begin{equation}\label{s1comdia}
\begin{matrix}  \caln^3\times S^3 &&& \\
                          &\searrow && \\
                          \decdnar{} && M^5_{k_1,k_2} &\\
                          & \swarrow && \\
                          T^2\times S^2 &&&,
\end{matrix}
\end{equation}
where the vertical arrow is the natural $T^2$-bundle projection map generated by the vector fields
\begin{equation}\label{Leqn}
L=\frac{1}{2k_1}\d_z-\frac{1}{2k_2}\xi_2,~\qquad \xi_{k_1,k_2}=\frac{1}{2k_1}\d_z+\frac{1}{2k_2}\xi_2.
\end{equation}
Here $\xi_2$ is the Reeb vector field of $\eta_2$ the standard contact form on $S^3$. The vector field $L$ generates the circle action of the southeast arrow, and $\xi_{k_1,k_2}$ generates the circle action of the southwest arrow, and it is the Reeb vector field of $\eta_{k_1,k_2}$. Note that on $\caln^3\times S^3$ the 1-form $\eta_{k_1,k_2}$ takes the form
\begin{equation}\label{1form}
\eta_{k_1,k_2}=k_1(dz-ydx) +k_2\eta_2.
\end{equation}

Now choose a compatible complex structure $\cJ_\grt$ on $T^2\times S^2$ as described in Section 4. We lift this to a complex structure $J_\grt$ in the contact bundle $\cald_{k_1,k_2}$. Since $(\gro_{k_1,k_2},\cJ_\grt)$ is K\"ahler, the lifted structure $\cals_{k_1,k_2}=(\12,\eta_{k_1,k_2},\Phi_\grt,g)$, where $\Phi_\grt$ extends $J_\grt$ by setting $\Phi_\grt\xi_{k_1,k_2}=0$, is Sasakian. If $\cX$ is a Hamiltonian Killing vector field that is holomorphic with respect to $\cJ_\grt$ then by Lemma \ref{liftham} it lifts to an infinitesimal automorphism $X$ of the Sasakian structure $\cals$.
If $J_\grt$ is in ${\mathcal S}_{0,\grt}$ or $S_{2m,\grt}$ for $m=1,\cdots,\lceil\frac{k_1}{k_2}\rceil-1$, then $X=X_{2m}$ is the induced Hamiltonian Killing vector field on $\m5$. Since the Reeb vector field is in the center of $\ga\gu\gt(\cals)$, the Sasaki cone $\grk(\cald_{k_1,k_2},J_\grt)$ has dimension two and is determined by
$$0<\eta_{k_1,k_2}(a\xi_{k_1,k_2}+bX_{2m})=\frac{a}{2}+\frac{a}{2}+b\eta_{k_1,k_2}(X_{2m})=a+k_2b\eta_2(X_{2m}).$$
This proves (1). 

For item (2) we see that Lemma \ref{A0lem} says that the complex structure $A_{0,\grt}$ has no Hamiltonian Killing vector fields. Thus, $M^5_{k_1,k_2}$ with this complex structure has no induced Hamiltonian circle action, and it follows that for $J\in A_{0,\grt}$ the Sasaki cone $\grk(\cald_{k_1,k_2},J)$ has dimension one.
\end{proof}

The Hamiltonian Killing vector fields referred to in this lemma are induced from the vector fields $H_i, i=1,2$ on $S^3$ that generate a maximal torus in the automorphism group $U(2)$ of the standard Sasakian structure on $S^3$. For future use we give these in terms of the standard coordinates on $\bbc^2$ 
\begin{equation}\label{Hieqn}
H_i=i(z_j\d_{z_j}-\bar{z}_j\d_{\bar{z}_j}).
\end{equation}

\section{Extremal Sasakian Structures}\label{extsassec}

The main result in this section involves lifting our extremal K\"ahler metrics to extremal Sasaki metrics via the Boothby-Wang construction. We then deform in the Sasaki cone to obtain quasiregular Sasakian structures which project to K\"ahler orbifolds which in turn we show have extremal representatives. Then the Openess Theorem of \cite{BGS06} shows that extremal structures exhaust the entire Sasaki cone. As in \cite{Boy10a,Boy10b} this will give rise to bouquets of extremal Sasakian structures. 

Given a Sasakian structure $\cals=(\xi,\eta,\Phi,g)$ on a compact manifold $M^{2n+1}$ we deform the contact 1-form by $\eta\mapsto \eta(t)=\eta+t\grz$ where $\grz$ is a basic 1-form with respect to the characteristic foliation $\calf_\xi$ defined by the Reeb vector field $\xi.$ Here $t$ lies in a suitable interval containing $0$ and such that $\eta(t)\wedge d\eta(t)\neq 0$. This gives rise to a family of Sasakian structures $\cals(t)=(\xi,\eta(t),\Phi(t),g(t))$ that we denote by ${\mathcal S}(\xi, \bar{J})$ where $\bar{J}$ is the induced complex structure on the normal bundle $\nu(\calf_\xi)=TM/L_\xi$ to the Reeb foliation $\calf_\xi$ which satisfy the initial condition $\cals(0)=\cals$. On the space ${\mathcal S}(\xi, \bar{J})$  we consider the ``energy functional'' $E:{\mathcal S}(\xi,\bar{J})\ra{1.4} \bbr$ defined by
\begin{equation}\label{var}
E(g) ={\displaystyle \int _M s_g ^2 d{\mu}_g ,}\, 
\end{equation}
i.e. the $L^2$-norm squared of the scalar curvature $s_g$ of the Sasaki metric $g$. Critical points $g$ of this functional are called {\it extremal Sasakian metrics}.  Similar to the K\"ahlerian case, the Euler-Lagrange equations for this functional says \cite{BGS06} that $g$ is critical if and only if the gradient vector field $J{\rm grad}_gs_g$ is transversely holomorphic, so, in particular, Sasakian metrics with constant scalar curvature are extremal. Since the scalar curvature $s_g$ is related to the transverse scalar curvature $s^T_g$ of the transverse K\"ahler metric by $s_g=s_g^T-2n$, a Sasaki metric is extremal if and only if its transverse K\"ahler metric is extremal. Hence, in the regular (quasi-regular) case, an extremal K\"ahler metric lifts to an extremal Sasaki metric, and conversely an extremal Sasaki metric projects to an extremal K\"ahler metric. 

Note that the deformation $\eta\mapsto \eta(t)=\eta+t\grz$ not only deforms the contact form, but also deforms the contact structure $\cald$ to an equivalent (isotopic) contact structure. So when we say that a contact structure $\cald$ has an extremal representative, we mean so up to isotopy. Deforming the K\"ahler form within its K\"ahler class corresponds to deforming the contact structure within its isotopy class. It is convenient to restrict the class of isotopies. Let $N(\gT)$ denote the normalizer of the maximal torus $\gT$ in $\gC\gr(\cald,J)$, and let ${\mathcal S}^{N(\gT)}(\xi, \bar{J})$ denote the subset of ${\mathcal S}(\xi, \bar{J})$ such that the basic 1-form $\grz$ is invariant under $N(\gT)$. We refer to an isotopy obtained by a deformation of the contact structure with $\grz\in {\mathcal S}^{N(\gT)}(\xi, \bar{J})$ as an {\it $N(\gT)$-isotopy} and we denote the $N(\gT)$-isotopy class of such contact structures by $\bar{\cald}$. As in Lemma 2.3 of \cite{Boy10b} it follows from a theorem of Calabi \cite{Cal85} that any extremal representative will lie in ${\mathcal S}^{N(\gT)}(\xi, \bar{J})$. Notice also that under a choice of isomorphism $\nu(\calf_\xi)\approx \cald$, $\bar{J}$ is isomorphic to $J$. Next we have

\begin{lemma}\label{sasconelem}
The Sasaki cone $\grk(\cald,J)$ depends only on the $N(\gT)$-isotopy class $\bar{\cald}$.
\end{lemma}

\begin{proof}
Let $(\cald,J)$ and $(\cald',J')$ be connected by an $N(\gT)$-isotopy. Then the Lie algebras $\gt$ and $\gt'$ actually coincide since they are maximal and invariant under $N(\gT)$. Thus, the unreduced Sasaki cones $\gt^+$ and $(\gt')^+$ coincide. But all vector fields in these cones are also invariant under the Weyl group $\calw=N(\gT)/\gT$. Thus, the reduced Sasaki cones $\grk(\cald,J)$ and $\grk(\cald',J')$ coincide.
\end{proof}

This lemma allows us to say that a contact structure $\cald$ has an extremal representative by which we mean that it is $N(\gT)$-isotopic to a contact structure with a compatible extremal Sasakian metric.
Notice also that under a transverse homothety extremal Sasakian structures stay extremal, and a transverse homothety of a CSC Sasakian structure is also a CSC Sasakian structure. This is because under the transverse homothety $\cals\mapsto \cals_a$ the scalar curvature of the metric $g_a$  is given by (cf. \cite{BG05}, page 228) $s_{g_a}=a^{-1}(s_g+2n)-2n$. 

In the case where the complex structure $J$ comes from ${\mathcal S}_{0,\grt}\sqcup_mS_{2m,\grt}$, we see from Lemma \ref{sasconeprop} that for any element of the Sasaki cone $\grk(\cald_{k_1,k_2},J)$ we must have $0<a+bk_2|z_2|^2$. In this case the Sasaki cone is determined by the conditions 
\begin{equation}\label{prodcond}
a>0,\qquad  a+bk_2>0.
\end{equation}

Similar to Section \ref{comsurf}, let $S_n$ denote ${\mathbb P}({\mathcal O} \oplus L) \rightarrow T^2$, where $L$ is a  holomorphic line bundle on $T^2$ of  degree $n$. We now have
\begin{lemma}\label{quotientlemma}
If $a$ and $b$ are integers with $b\neq 0$ satisfying conditions (\ref{prodcond}), the quotient of $(M^5_{k_1,k_2},J)$ by the circle action generated by the Reeb vector field $R_{ab}=a\12+bH_2$ is 
$(S_{n},\grD_{pq})$ with branch divisor
$$\grD_{pq}=(1-\frac{1}{p})E_{n} +(1-\frac{1}{q})E_\infty,$$
where $p=a$,  $q=a+k_2b$, and $n$ is some integer determined by  $J$, $p$, $q$, $k_1$ and $k_2$.
\end{lemma}
\begin{proof}
We want to identify the quotient of $\caln\times S^3$ by the 2-torus generated by the vector fields $L$ and $R_{ab}$ given in coordinates by
\begin{equation}\label{Leqn2}
L=\frac{1}{2k_1}\d_z-\frac{1}{2k_2}\xi_2,~\qquad R_{ab}=\frac{a}{2k_1}\d_z+\frac{a}{2k_2}H_1+\frac{a+2k_2b}{2k_2}H_2
\end{equation}
where $H_i$ is given by Equation (\ref{Hieqn}). 

The quotient by the circle generated by $L$ is $M^5_{k_1,k_2}$. So we have the commutative diagram
\begin{equation}\label{s1comdia2}
\begin{matrix}  \caln^3\times S^3 &&& \\
                          &\searrow && \\
                          \decdnar{} && M^5_{k_1,k_2} &\\
                          & \swarrow && \\
                          B_{a,b}&&&
\end{matrix}
\end{equation}
where by \cite{BG00a} $B_{a,b}$ is a projective algebraic orbifold, and the southwest arrow is the projection of the $S^1$-orbibundle generated by $R_{ab}$. For $(x,y,z;z_1,z_2)\in \caln\times S^3$ we let $[x,y,z;z_1,z_2]$ denote the corresponding class in $B_{a,b}$ under the $T^2$ action. If $(x',y',z';z'_1,z'_2)$ is another point of $\caln\times S^3$ representing the same class, Equation (\ref{Leqn2}) implies that $x'=x$ and $y'=y$. Thus, we have a natural projection $\grr:B_{a,b}\ra{1.6} T^2$. Moreover, one easily sees that 
\begin{equation}\label{RabmodL}
R_{ab}\equiv\frac{1}{k_2}(pH_1+qH_2) \mod \cali_L
\end{equation}
where $\cali_L$ denotes the ideal generated by $L$. So the fibers of $\grr$ are identified with the quotient of $S^3$ by the corresponding weighted circle action generated by $R_{ab}$ on $M^5_{k_1,k_2}$, namely the weighted projective line $\bbc\bbp(p,q)$. Moreover, the orbifold singular locus comes from the two points $[1,0]$ and $[0,1]$ in the fibers  $\bbc\bbp(p,q)$  with isotropy $\bbz_p$ and $\bbz_q$, respectively. The former is identified with zero section $E_n$ and the latter with the infinity section $E_\infty$. Thus, the orbifold $B_{a,b}$ can be represented as a weighted projectization of a rank two vector bundle $\calo\oplus L$ fibering over $T^2$ with fibers $\bbc\bbp(p,q)$  where the degree of $L$ is an integer $n$, that is, $B_{a,b}$ is identified with the $nth$ orbifold pseudo-Hirzebruch surface $B_{a,b}=(S_n,\grD_{pq})$.
\end{proof}

\begin{remark}
One can actually determine $n$ in terms of the integers $p,q,k_1,k_2$, but we do not need it here. We shall do so in the sequel \cite{BoTo12} at least in the case that $k_2=1$.
\end{remark}

Since all the complex structures in ${\mathcal S}_{0,\grt}\sqcup_mS_{2m,\grt}$ admit extremal K\"ahler metrics in every K\"ahler class, we
know that the regular Sasakian structures also admit extremal Sasaki metrics, and only for complex structures in ${\mathcal S}_{0,\grt}$ do we get CSC Sasaki metrics. We also know from \cite{BuBa88,Fuj92} that the one dimensional Sasaki cone associated to the non-split complex structure admits no extremal Sasaki metric. Thus, in order to complete the proof of Theorem \ref{mainthm}  we need to further consider the complex structures in ${\mathcal S}_{0,\grt}$ and $S_{2m,\grt}$ for $m=1,\cdots, \lceil\frac{k_1}{k_2}\rceil-1$ for the non-regular elements of the Sasaki cones. Essentially we need to establish
\begin{theorem}\label{existence2DSasakicone}
For the transverse complex structure \newline
$J\in {\mathcal S}_{0,\grt}\sqcup_{m=1,\cdots,\lceil\frac{k_1}{k_2}\rceil-1}S_{2m,\grt}$ on $M^5_{k_1,k_2}$ every member of the Sasaki cone $\grk(\cald_{k_1,k_2},J)$ admits extremal Sasaki metrics.
\end{theorem}

\begin{proof}
By the Openness Theorem of \cite{BGS06} it is enough to show that every quasi-regular ray in the Sasaki cone admits extremal Sasaki metrics and further by the homothety invariance of extremality we just need to show that one quasi-regular Sasaki structure in each quasi-regular ray admits an extremal Sasaki metric. In the next two subsections we will first determine the orbifold K\"ahler quotients of such Sasaki structures 

\subsection{The $m=0$ case}

When the complex structure is $J\in {\mathcal S}_{0,\grt}$ on $T^2\times_\grr \bbc\bbp^1 =\calz$, Suwa \cite{Suw69} shows that for $\grr\neq {\rm id}$, $H^0(\calz,\Theta_\calz)$ has dimension two. In a local coordinate chart $(w,\grz)$ on $U\times \bbc\bbp^1$ with $U=T^2-\{p_1,p_2\}$ where $p_1,p_2$ are distinct points of $T^2$, a basis for $H^0(\calz,\Theta_\calz)$ takes the form
$$\grz\d_\grz, \qquad \d_w+\bigl({\mathcal Z}(w-p_1)-{\mathcal Z}(w-p_2)\bigr) \grz\d_\grz,$$
where $\grz$ is an affine coordinate on $\bbc\bbp^1$ and ${\mathcal Z}$ is the Weierstrass zeta function. Neither the real nor imaginary part of the second vector field is Hamiltonian; however, the imaginary part of $\grz\d_\grz$ is a Hamiltonian Killing vector field and generates a Hamiltonian circle action. If we take the vector field $\cK=4{\rm Im}~\grz\d_\grz$ we see that the Hamiltonian is $H=\frac{1}{1+|\grz|^2}$. In homogeneous coordinates $(z_1,z_2)$ with $\grz=\frac{z_1}{z_2}$ on $\bbc\bbp^1$ we have $H=\frac{|z_2|^2}{|z_1|^2+|z_2|^2}$, so pulling back to $M^5_{k_1,k_2}$ gives $\pi^*H=|z_2|^2$. So by Lemma \ref{liftham} the Hamiltonian Killing vector field on $M^5_{k_1,k_2}$ is $X_0=\cK^h+|z_2|^2\12$ mod the ideal generated by $L$. A straightforward calculation shows that $X_0=H_2$ and this is also the Hamiltonian Killing vector field for $\grr={\rm id}$. Thus, from Lemma \ref{quotientlemma} the quotient orbifolds in the quasi-regular case have the form $(S_n,\grD_{pq})$ for some integer $n$.

\subsection{The $m>0$ case}

The discussion begins as in the degree zero case However, in this case the dimension of $H^0(\calz,\Theta_\calz)$ is $2m+1$ \cite{Suw69} and in local coordinates $(u,\grz)$ is spanned by $\grz\d_\grz,~u^{2m}\grz^2\d_\grz$ together with $2m$ holomorphic vector fields involving the Weierstrass $\wp$ function and its first $2(m-1)$ derivatives. As before $\cK=4{\rm Im}\grz\d_\grz$ generates the Hamiltonian circle action which in terms of homogeneous coordinates on the fibers is  $(w,[z_1,z_2])\mapsto (w,[z_1,\grl z_2])$. With respect to this complex structure the K\"ahler structure on $T^2\times S^2$ is taken to be that described explicitly by Equations (\ref{metric})-(\ref{positivity}) in Section \ref{exKmet}. Moreover, the complex structure lifts to a transverse complex structure and CR structure $(\cald_{k_1,k_2},J_{2m})$ on $M^5_{k_1,k_2}$, and it follows from Equation (\ref{RabmodL}) and the definition of $p$ and $q$ that the induced Hamiltonian Killing vector field on $M^5_{k_1,k_2}$ is again represented by $H_2$. Thus, the Sasaki cone is exactly the same as in the previous case. It is represented by the set of all elements in $\gt_2^+$ of the form $a\12+bH_2$  where $a,b$ satisfy conditions (\ref{prodcond}). As in Lemma \ref{quotientlemma} we set $p=a$ and $q=a+bk_2$. For $p$ and $q$ positive integers the Reeb vector field $a\12+bH_2=pH_1+qH_2 \mod \cali_L$  generates a locally free $S^1$ action, and by Lemma \ref{quotientlemma} its quotient is the orbifold 
$(S_{n},\grD_{pq})$ for some integer $n$ now determined by  $m$ as well as $p$, $q$, $k_1$ and $k_2$.

\subsection{Existence of Extremal K\"ahler metrics on $(S_{n},\grD_{pq})$}

To finish the proof of Theorem \ref{existence2DSasakicone} we will now show that for any
$n \in {\mathbb Z}$ and any pair of co-prime positive integers $p,q$, every K\"ahler class on 
$(S_{n},\grD_{pq})$ admits an extremal K\"ahler metric. Without loss of generality we may assume $n \geq 0$.

First let us consider the case $n=0$. From Section \ref{comsurf} we know that
$S_n$ may be described
as  $T^2\times_\grr \bbc\bbp^1$, where $\rho$ lies in a circle. Thus 
$(S_{n},\grD_{pq})$ may be viewed as the as a complex orbifold $T^2\times_\grr\bbc\bbp(p,q)$, where $\bbc\bbp(p,q)$ denotes the weighted projective space. 
As algebraic varieties $T^2\times_\grr\bbc\bbp^1$ and $T^2\times_\grr\bbc\bbp(w_1,w_2)$ are isomorphic for each homomorphism $\grr$; however, the latter has branch divisors along $E_0$ and $E_\infty$ with ramification index $p$ and $q$, respectively, making them inequivalent as orbifolds. It is often convenient to view such orbifolds as pairs; the former has trivial orbifold structure and is written as $(T^2\times_\grr\bbc\bbp^1,\emptyset)$, whereas the latter is written as $(T^2\times_\grr\bbc\bbp^1,\grD)$ with branch divisor
$$\grD=(1-\frac{1}{p})E_0 +(1-\frac{1}{q})E_\infty.$$
Then the identity map (as sets) 
$$\BOne:(T^2\times_\grr\bbc\bbp^1,\grD)\ra{2.3} (T^2\times_\grr\bbc\bbp^1,\emptyset)$$ 
is a Galois covering map with trivial Galois group, and the inequivalent orbifolds $(T^2\times_\grr\bbc\bbp^1,\emptyset)$ and $(T^2\times_\grr\bbc\bbp^1,\grD)$ have the same underlying complex structure. The employment of such Galois orbifold covers originated in \cite{GhKo05} and was subsequently exploited in \cite{Boy11,BoPa10}. We will exploit this point of view to treat the case $n>0$ as well.

The scalar curvature of the $T^2$ factor vanishes, and for weighted projective spaces it was computed  in \cite{BGS06}. Moreover, weighted projective spaces are known \cite{Bry01,DaGa06} to admit extremal K\"ahler metrics. Hence, $T^2\times_\grr\bbc\bbp(p,q)$ admits (local) product extremal K\"ahler metrics in every K\"ahler class which have constant scalar curvature 
if and only if $p=q$. 

Now let us take care of the case $n>0$. Here again we have the identity map
$$\BOne:(S_n,\grD)\ra{2.3} (S_n,\emptyset).$$
The extremal K\"ahler metrics for $(S_{n},\emptyset)$ were given in Section \ref{exKmet} where the ``canonical'' structure is determined by taking $\Theta(\gz)=1-\gz^2$. To describe the extremal metrics for $(S_{n},\grD_{pq})$ we adobt the discussion in Section 1.3 of \cite{ACGT08} to the orbifold setting as in Section 1.5 of \cite{ACGT04}. The ``canonical'' structure (giving the Abreu-Guillemin metric on the fibers, cf. \cite{Abr01}) is determined by the function 
\begin{equation}\label{thetac}
\Theta_c (\gz) = \frac{2pq(1+\gz)(1-\gz)}{p^2q(1-\gz) + q^2p(1+\gz)},
\end{equation}
where we assume that $p$ and $q$ are relatively prime positive integers.
Now the conditions \eqref{positivity} on $\Theta$ are replaced by
\begin{align}
\label{conesingpositivity}
(i)&~ \Theta({\gz}) > 0, \quad -1 < \gz <1,\quad
(ii)\ \Theta(\pm 1) = 0,\quad \\ \notag
(iii)& ~\Theta'(-1) =  2/p, \quad \Theta'(1) = -2/q.
\end{align}
It is easy to check that $\Theta_c$ of Equation \ref{thetac} satisfies these conditions. However, the metric obtained from $\Theta_c$ is not extremal. We obtain extremal K\"ahler metrics from Proposition 1 in \cite{ACGT08} by defining the function
$F(\gz) = \Theta(\gz) (1+ r \gz)$.  We see that $g$ is extremal exactly when
$F(\gz)$ is a polynomial of degree at most 4 and $F''(-1/r) = 0$. Together with the endpoint conditions of \eqref{conesingpositivity} this implies that $F(\gz)$ must be given by
\begin{equation}\label{orbiextremalpolynomial}
F(\gz) = \frac{(1-\gz^2)h(\gz)}{4 p q (3 - r^2)},
\end{equation}
where 
\[
\begin{array}{ccl}
h(\gz) & = & q (6 - 3 r - 4 r^2 + r^3) + p (6 +  3 r - 4 r^2 - r^3)\\
\\
 &+& 2 (3 - r^2) (q (r-1) + p (1 + r))\gz \\
 \\
 &+ &  r(p (3 + 2 r - r^2) - q (3 - 2 r - r^2))\gz^2,
 \end{array}
    \]
and $-1<\gz<1$. We can then check that $\Theta(\gz)$ as defined via $F(\gz)$ above satisfies all the conditions of \eqref{conesingpositivity}. Thus for any pair $(p,q)$ of positive integers and for all $r \in (0,1)$ we have an extremal K\"ahler metric. Since up to rescaling the set $\{r \in (0,1)\}$ still exhausts the K\"ahler cone, we are done.
\end{proof}

\begin{remark} Notice that the parameter $r$ is also determined via the Sasakian quotient by $m$,  $p$, $q$, $k_1$ and $k_2$. Experimental data indicate that for some choices of these data, we  arrive at an $r$ such that $h(\gz)$ is a linear function and hence, according to \cite{ACGT08}, the corresponding K\"ahler and Sasakian extremal K\"ahler metrics have constant scalar curvature. We treat this issue carefully in \cite{BoTo12}.
Note that here $F'(\gz)$ does not have a double root at $\gz=-1/r$, which is the criterion for a K\"ahler-Einstein metric in this set-up (see e.g. Section 3 of \cite{ACGT08b}). There are no K\"ahler-Einstein metrics on $T^2\times S^2$ nor regular Sasaki-Einstein metrics on $T^2\times S^3$.
\end{remark}

\begin{ack}
The authors would like to thank Vestislav Apostolov, David Calderbank, and Paul Gauduchon for helpful conversations, and an anonymous referee for making suggestions that enhance the clarity of the exposition.
We would also like to thank the Simons Center at Stony Brook for financial support during a short visit there in March 2011 where work on this paper began.
\end{ack}

\def\cprime{$'$} \def\cprime{$'$} \def\cprime{$'$} \def\cprime{$'$}
  \def\cprime{$'$} \def\cprime{$'$} \def\cprime{$'$} \def\cprime{$'$}
  \def\cdprime{$''$} \def\cprime{$'$} \def\cprime{$'$} \def\cprime{$'$}
  \def\cprime{$'$}
\providecommand{\bysame}{\leavevmode\hbox to3em{\hrulefill}\thinspace}
\providecommand{\MR}{\relax\ifhmode\unskip\space\fi MR }
\providecommand{\MRhref}[2]{%
  \href{http://www.ams.org/mathscinet-getitem?mr=#1}{#2}
}
\providecommand{\href}[2]{#2}

\end{document}